\documentclass[10pt,a4paper]{article}

\usepackage{a4wide}

\usepackage{graphicx,xcolor,textpos}
\usepackage{helvet}
\usepackage[utf8]{inputenc}
\usepackage[english]{babel}
\usepackage[T1]{fontenc}
\usepackage{gensymb}
\usepackage{float}
\usepackage{caption}
\usepackage{subcaption}
\usepackage{amsmath, mathtools, mathrsfs}
\usepackage{amssymb}
\usepackage{epstopdf}
\usepackage{amsthm}
\usepackage{enumitem}
\usepackage[makeroom]{cancel}
\usepackage{hyperref}
\usepackage{tikz}
\usepackage{tikz-cd}
\usepackage{blkarray}
\usepackage[framemethod = default]{mdframed}
\usepackage{todonotes}
\usepackage{faktor}

\newtheorem{theorem}{Theorem}[section]
\theoremstyle{definition} 
\newtheorem{definition}[theorem]{Definition}
\theoremstyle{plain}
\newtheorem{prop}[theorem]{Proposition}
\theoremstyle{definition} 

\theoremstyle{plain}
\newtheorem{lemma}[theorem]{Lemma}
\newtheorem{cor}[theorem]{Corollary}

\theoremstyle{definition}

\DeclareMathOperator{\GL}{GL}

\DeclareMathOperator{\Aut}{Aut}

\DeclareMathOperator{\gal}{Gal}
\DeclareMathOperator{\rank}{rank}
\DeclareMathOperator{\End}{End}
\DeclareMathOperator{\orb}{orb}
\DeclareMathOperator{\stab}{stab}

\newcommand{\Z}{\mathbb{Z}}
\newcommand{\R}{\mathbb{R}}
\newcommand{\C}{\mathbb{C}}
\newcommand{\Q}{\mathbb{Q}}
\newcommand{\N}{\mathbb{N}}

\newcommand{\g}{\mathfrak{g}}
\newcommand{\h}{\mathfrak{h}}

\newcommand{\n}{\mathfrak{n}}
\newcommand{\m}{\mathfrak{m}}

\begin{document}
	\title{On closed manifolds admitting an Anosov\\diffeomorphism but no expanding map}
	\author{Jonas Der\'e and Thomas Witdouck\thanks{Email: \href{mailto:jonas.dere@kuleuven.be}{jonas.dere@kuleuven.be} The first author was supported by a postdoctoral fellowship of the Research Foundation -- Flanders (FWO). The second author was supported by long term structural funding - Methusalem grant of the Flemish Government.}}
	\date{\today}
	
	\maketitle
	
	\begin{abstract}
A few years ago, the first example of a closed manifold admitting an Anosov diffeomorphism but no expanding map was given. Unfortunately, this example is not explicit and is high-dimensional, although its exact dimension is unknown due to the type of construction. In this paper, we present a family of concrete $12$-dimensional nilmanifolds with an Anosov diffeomorphism but no expanding map, where nilmanifolds are defined as the quotient of a $1$-connected nilpotent Lie group by a cocompact lattice. We show that this family has the smallest possible dimension in the class of infra-nilmanifolds, which is conjectured to be the only type of manifolds admitting Anosov diffeomorphisms up to homeomorphism. The proof shows how to construct positive gradings from the eigenvalues of the Anosov diffeomorphism under some additional assumptions related to the rank, using the action of the Galois group on these algebraic units.
\end{abstract}

\section{Introduction}
Expanding maps and Anosov diffeomorphisms form an important class of dynamical systems as they combine structural stability with chaotic behavior. These maps were intensively studied at the end of the sixties by S.~Smale, J.~Franks and M.~Shub in the series of papers \cite{fran69-1,fran70-1,shub69-1,shub70-1,smal67-1}. One of the central research problems is to describe the closed manifolds admitting such a dynamical system. The main goal of this paper is to present a family of 12-dimensional nilmanifolds admitting an Anosov diffeomorphism but no expanding map and to show that this is the lowest possible dimension for such an example. From this example we construct families with the same properties in every possible dimension $\geq 14$ as well. 

By the work of M. Gromov in \cite{grom81-1}, it is known that every closed manifold admitting an expanding map is homeomorphic to an infra-nilmanifold, i.~e.~the compact quotient of a nilpotent Lie group $N$ by a discrete group of isometries. Not every infra-nilmanifold admits expanding maps though, with the first classical example due to \cite{dl57-1}. The question which infra-nilmanifolds do admit an expanding map was answered by K. Dekimpe and the first author in \cite{dd14-1,dere14-1}, by showing that it only depends on the existence of a positive grading on the Lie algebra corresponding to the covering nilpotent Lie group $N$. For the manifolds admitting an Anosov diffeomorphism on the other hand, a lot less is known. It is conjectured that, up to homeomorphism, only infra-nilmanifolds can admit an Anosov diffeomorphism, but this has been open ever since the original question in \cite{smal67-1}. The conjecture has been confirmed only in some special cases, for example for codimension one Anosov diffeomorphisms in \cite{newh70-1}. Also describing the infra-nilmanifolds admitting an Anosov diffeomorphism is a challenging question, on which many people have worked over the past decades. 

Note that the first example of a nilmanifold with an Anosov diffeomorphism that does not admit an expanding map was given in \cite{dere14-1}. Unfortunately this example is not explicit and its dimension is unknown. This is due to the nature of the construction, since the manifold is constructed from a quotient of a free nilpotent Lie algebra by an ideal defined from a generating set with specific elements, making it hard to compute its dimension. In Section \ref{sec:familyofexamples} we improve this result by giving a family of $12$-dimensional concrete nilmanifolds having an Anosov diffeomorphism but no expanding map, using the terminology introduced in Section \ref{sec:prelim}. As a consequence of how being Anosov or having a positive grading behaves under taking direct sums of Lie algebras, as demonstrated in Section \ref{sec:decom}, we also give examples in every dimension $\geq 14$. 

The main result of this paper is not only the existence of the family in dimension $12$, but also the proof that this is the smallest possible dimension for such an example in the class of infra-nilmanifolds. In order to show this, one possible strategy could be to classify all Anosov Lie algebras in lower dimensions and show that these all have a positive grading. In \cite{lw08-1, mw15-1}, a complete classification is given up to dimension $9$, but it is hard to extend this classification any further. Our approach is different as we construct a positive grading from the eigenvalues of the Anosov automorphism in a more general setting, without completely determining the Lie bracket. 

The methods we develop use the properties of the eigenvalues of these automorphism, which are hyperbolic algebraic integers. A first study of how these properties determine the structure of the Lie algebra was given in \cite{main12-1}, although we shift our focus to the action of the Galois group on these eigenvalues, making our results more general. In particular, the notion of rank for Anosov automorphisms plays a crucial role, as introduced in Section \ref{sec:rank}, giving bounds on the dimension of the lower central series of an Anosov Lie algebra. This is partially based on some properties introduced in \cite{payn09-1}, although the focus of the latter paper was on quotients of free Lie algebras, whereas we start from a given type for a Lie algebra. The final result is then presented in Section \ref{sec:dimension}, by checking the possibilities for every type.

\section{Preliminaries}
\label{sec:prelim}

In this section, we give all the necessary background for the remainder of the paper. The first part deals with the theory of infra-nilmanifolds, in particular the relation to expanding maps and Anosov diffeomorphisms, whereas the second part discusses the basics of of Galois theory.

\subsection{Infra-nilmanifolds}

We recall the basis properties of Anosov diffeomorphisms and expanding maps on (infra-)nilmanifolds, for more details we refer to \cite{deki18-1}. 

Let $N$ be a simply connected and connected nilpotent Lie group with corresponding real nilpotent Lie algebra $\n^\R$. We define a \emph{nilmanifold} as a quotient $\Gamma \backslash N$ with $\Gamma$ a uniform lattice in $N$. In this case, the group $\Gamma$ is a finitely generated torsion-free nilpotent group and vice versa, every group with these properties uniquely embeds into a simply connected and connected nilpotent Lie group as a uniform lattice. There is a unique rational Lie algebra $\n^\Q$ corresponding to $\Gamma$, defined as the rational span of $\log(\Gamma)$ in $\n^\R$. In general, we call a rational subalgebra $\n^\Q$ of $\n^\R$ such that $\dim_\Q(\n^\Q) = \dim_\R(\n^\R)$ and the $\R$-span of $\n^\Q$ is $\n^\R$ a \emph{rational form} of a real Lie algebra. Every rational form of $\n^\R$ corresponds to a lattice $\Gamma$ of $N$. Moreover, if $\Gamma_1$ and $\Gamma_2$ are two lattices corresponding to the same rational form $\n^\Q$, then the lattices $\Gamma_1$ and $\Gamma_2$ are commensurable, meaning that $\Gamma_1 \cap \Gamma_2$ has finite index in both $\Gamma_1$ and $\Gamma_2$. 

The existence of an Anosov diffeomorphism on the nilmanifold $\Gamma \backslash N$ only depends on the rational Lie algebra $\n^\Q$, more specifically on the existence of a certain automorphism. We first introduce some terminology. Let $E \subset \C$ be any subfield of the complex numbers and $\n^E$ a Lie algebra defined over $E$.
\begin{definition}
 A linear map on $\n^E$ is called \textit{hyperbolic} if it has no eigenvalues of absolute value 1 and is called \textit{integer-like} if its characteristic polynomial has integer coefficients and determinant $\pm 1$. We say that an automorphism $A:\n^E \to  \n^E$ is an \textit{Anosov automorphism} if is both hyperbolic and integer-like. A nilpotent rational Lie algebra $\n^\Q$ that admits an Anosov automorphism is called an \textit{Anosov Lie algebra}.
\end{definition}   
By \cite{deki99-1}, the nilmanifold $\Gamma \backslash N$ admits an Anosov diffeomorphism if and only if the Lie algebra $\n^\Q$ is Anosov. For expanding maps on the other hand, the existence depends only on the covering Lie group $N$.
\begin{definition}
We call an automorphism $A: \n^E \to \n^E$ \textit{expanding} if every eigenvalue of $A$ has absolute value $> 1$. We say that $\n^E$ has a \textit{positive grading} if there exists a decomposition of $\n^E = \displaystyle \bigoplus_{i > 0} \n_i$ into subspaces $\n_i \subset \n^E$ such that $[\n_i,\n_j] \subset \n_{i+j}$.
\end{definition}
By \cite[Theorem 3.1.]{dere14-1} the existence of an expanding map on $\Gamma \backslash N$ is equivalent to the existence of an expanding automorphism on $\n^\R$, which is equivalent to the existence of a positive grading on $\n^\R$. If $\n^E$ is a Lie algebra over the field $E$ and $E \subset F$ is a field extension, then the Lie algebra $\n^F = \n^E \otimes_E F$ by extending scalar multiplication is a Lie algebra of the same dimension over $E$. Note that this notation is compatible with the previous one for rational forms $\n^\Q$ in a real Lie algebra $\n^\R$. The existence of a positive grading does not depend on the field $E$, meaning that $\n^E$ has a positive grading if and only if $\n^F$ has a positive grading. In order to give the family of nilmanifolds it suffices to give a family of $12$-dimensional rational Anosov Lie algebras that have no expanding automorphism, which we will do in Section \ref{sec:familyofexamples}.

Note that there is the more general notion of \emph{infra-nilmanifolds}, which we do not introduce in full detail here. Every infra-nilmanifold is finitely covered by a nilmanifold, and if the infra-nilmanifold admits an Anosov diffeomorphism, then also the covering nilmanifold admits an Anosov diffeomorphism. Moreover, if the covering nilmanifold does not admit an expanding map, then the infra-nilmanifold does not have an expanding map, see \cite{dere14-1}. Hence in order to show that our family has minimal dimension in the class of infra-nilmanifolds, it suffices to show that it has minimal dimension in the class of nilmanifolds. The remainder of this paper only deals with Lie algebras and their properties, where the dynamical implications follow from the discussion in this section.

\subsection{Galois theory}
\label{sec:prelimGalois}
In this section we recall some basic facts from Galois theory based on the book \cite{stew98-1} and the relation with Anosov Lie algebras. 

Let $E/F$ be a field extension of finite degree. The automorphism group of this extension, denoted by $\Aut(E, F)$, is defined as the group of all field automorphisms of $E$ that fix $F$, i.e.
\begin{align*}
    \Aut(E, F) = \{ \sigma \in \Aut(E) \mid \forall x \in F: \sigma(x) = x \}.
\end{align*}
An extension $E/F$ is called \emph{Galois} if the points that are fixed by all elements of $\Aut(E, F)$ are exactly the elements of the smaller field $F$. We write $\gal(E, F)$ for the automorphisms of a Galois extension $E/F$ in order to emphasize that the extension is Galois. 

Now fix a polynomial $f$ of degree $r$ with coefficients in $\Q$. The splitting field of $f$ over $\Q$ is the unique smallest field $\Q(f)$ with $\Q \subset \Q(f) \subset \C$ such that $\Q(f)$ contains all the roots of $f$. The splitting fields over $\Q$ are exactly the Galois extensions of $\Q$ and most of the time we will be working with these fields. Now let $g$ be another rational polynomial with $\Q(g) \subset \Q(f)$, then $\Q(f)/\Q(g)$ is Galois as well. Moreover, as a consequence of the fundamental theorem of Galois theory, $\gal(\Q(f), \Q(g))$ is a normal subgroup of $\gal(\Q(f), \Q)$. For any root $\lambda$ of $g$ and automorphism $\sigma \in \gal(\Q(f), \Q)$, the element $\sigma(\lambda)$ is again a root of $g$. This gives an action of $\gal(\Q(f), \Q)$ on the set of roots of $g$. Moreover, if $g$ is irreducible of degree $n$ all its roots $\lambda_1, \ldots, \lambda_n$ are distinct and the action of the Galois group is transitive, i.e.~for any two roots $\lambda_i, \lambda_j$, there exists an automorphism $\sigma \in \gal(\Q(f), \Q)$ such that $\sigma(\lambda_i) = \lambda_j$. The action is determined by the map

$$ \iota_g:\gal(\Q(f), \Q) \to S_n: \sigma \mapsto \iota(\sigma) $$
	with
$$ \iota(\sigma)(i) = j \Leftrightarrow \sigma(\lambda_i) = \lambda_j. $$
In general, $f$ is not irreducible, but it can be written as a product of its irreducible factors over $\Q$, say $f = g_1 \cdot \ldots \cdot g_m$ with $\deg(g_i) = r_i$. For each factor $g_i$ we have that $\Q(g_i) \subset \Q(f)$. Note that since $r = r_1 + \ldots + r_m$, we get the natural inclusion $S_{r_1} \oplus \ldots \oplus S_{r_m} \hookrightarrow S_{r}$ 
This gives a map
\begin{equation*}
    \begin{tikzcd}
    \iota:\gal(\Q(f), \Q) \arrow[rr, "\iota_{g_1} \oplus \ldots \oplus \iota_{g_m}"] & & S_1 \oplus \ldots \oplus S_{r_m} \arrow[r, hookrightarrow] & S_r.
    \end{tikzcd}
\end{equation*}
This map is injective as a field automorphism of $\Q(f)$ is completely determined by its restriction to the roots of $f$. During this paper we will, after ordering the roots of the polynomial $f$, drop the `$\iota$' and just write $\sigma$ for both the field automorphism and the corresponding permutation in $S_r$, as there is no confusion possible. We will write elements of $S_r$ with standard cycle notation, i.e. if $\sigma = (a_1\, a_2 \, \ldots \, a_s)$, then $\sigma(a_1) = a_2, \, \sigma(a_2) = a_3, \, \ldots \, , \sigma(a_s) = a_1$ and $\sigma$ fixes all other elements.

It is natural to apply the orbit-stabilizer theorem to the action of the Galois group on the roots.

\begin{theorem}[Orbit-stabilizer]
    Let $G$ be a group which has an action on a set $X$. For any $x \in X$, we have that
    \begin{align*}
        |\orb_G(x)|\cdot |\stab_G(x)| = |G|.
    \end{align*}
    where $\orb_G(x) = \{ g \cdot x \mid g \in G \}$ is the orbit of $x$ under the action of $G$ and $\stab_G(x) = \{ g \in G \mid g\cdot x = x \}$ is the stabilizer subgroup of $x$ in $G$.
\end{theorem}

A consequence that we will use often is stated below.

\begin{lemma}
    \label{lem:IrrPrimeDegreeCyclicPerm}
    Let $f$ be a irreducible polynomial over $\Q$ of prime degree $p$ with splitting field $E$ over $\Q$. There exists an element $\sigma \in \gal(E,\Q)$ of order $p$ and an ordering of the roots such that $\sigma = (1 \, 2 \, \ldots \, p)$.
\end{lemma}
\begin{proof}
    Since $f$ is irreducible, the action of $G := \gal(E, \Q)$ on the roots of $f$ is transitive. Therefore we get for any root $\lambda$ of $f$ that $|\orb_G(\lambda)| = p$. The orbit-stabilizer thus tells us that $p$ divides the order of $G$. Following Cauchy's theorem there must exist an element $\sigma \in G$ of order $p$. Next we apply the orbit-stabilizer theorem to the group generated by $\sigma$, which we denote by $\langle \sigma \rangle$. This gives for any root $\lambda$ of $f$ that $|\orb_{\langle \sigma \rangle}(\lambda)|\cdot |\stab_{\langle \sigma \rangle}(\lambda)| = p$ and thus that $|\orb_{\langle \sigma \rangle}(\lambda)| = 1$ or $|\orb_{\langle \sigma \rangle}(\lambda)| = p$. If the orbit under $\sigma$ of any root (and hence every root) is equal to one, then $\sigma$ is the identity map which is a contradiction because it is of order $p$. Therefore there exists a root $\mu$ of which the orbit under $\langle \sigma \rangle$ counts $p$ elements and thus contains all roots of $f$. Then define $\lambda_i := \sigma^i(\mu)$ for $1 \leq i \leq p$. It follows that with this ordering of the roots, $\sigma$ is equal to the permutation $(1 \,2 \, \ldots \, p)$.
\end{proof}

To conclude this section, we give a method from \cite{dere13-1} for constructing Anosov Lie algebras using the action of the Galois group. Let $E$ be some finite degree Galois extension of the rationals and $\n^\Q$ a rational Lie algebra. By choosing a basis for $\n^\Q$, we get a natural identification $\n^E \approx E^n$ and an action of the Galois group of $E/\Q$ on $\n^E \approx E^n$ by letting $v^\sigma = (v_1, \ldots, v_n)^\sigma = (\sigma(v_1), \ldots, \sigma(v_n))$ for all $v \in \n^E$ and $\sigma \in \gal(E, \Q)$. This action does not depend on the choice of basis for $\n^\Q$. Let us denote the set of all linear maps on $\n^E$ by $\End(\n^E)$. Then we get an action of $\gal(E, \Q)$ on $\End(\n^E)$ as well, by letting $A^\sigma(v) = (A(v^{\sigma^{-1}}))^\sigma$ for all $v \in \n^E$, $A \in \End(\n^E)$ and $\sigma \in \gal(E, \Q)$.
\begin{theorem}
\label{thm:constructingAnosov}
    Let $\n^\Q$ be a rational Lie algebra and $\rho:\gal(E, \Q) \to \Aut(\n^\Q)$ be a representation. Suppose there exists a Lie algebra automorphism $A:\n^E \to \n^E$ such that $\rho_\sigma A \rho_{\sigma^{-1}} = A^\sigma$ for all $\sigma \in \gal(E, \Q)$. Then there also exists a rational form $\m^\Q \subset \n^E$ such that $A$ induces an automorphism on $\m^\Q$. If all eigenvalues of $A$ are algebraic units of absolute value different from 1, then $A:\mathfrak{m}^\Q \to \mathfrak{m}^\Q$ is an Anosov automorphism.
\end{theorem}
\noindent This result is constructive as the rational form $\mathfrak{m}^\Q$ is explicitly given by
\begin{equation*}
    \mathfrak{m}^\Q = \{ v \in \n^E \mid\,\, \forall \sigma \in \gal(E, \Q):\,\rho_\sigma(v) = v^\sigma \}.
\end{equation*}
With this result we are ready to present the family of Anosov Lie algebras in dimension $12$ in the next section.



\section{Family of Anosov Lie algebras without positive grading}
\label{sec:familyofexamples}

In this section we exhibit a family of $12$-dimensional Anosov Lie algebras that do not have a positive grading. The family consists of rational Lie algebras $\m_k^\Q$ for $k > 0$ of type $(4,2,2,2,2)$ with $\m_k^\Q$ and $\m_l^\Q$ isomorphic if and only if $\frac{k}{l} = m^2$ for some integer $m \in \Z$. In order to show that the Lie algebras are Anosov, we construct these from Theorem \ref{thm:constructingAnosov} by using a fixed Lie algebra $\n^\Q$ and different fields $E$. To check that the Lie algebras do not have a positive grading we apply the methods of \cite{dere14-1}, whereas the isomorphisms between the Lie algebras $\m_k^\Q$ follow from a classification given in \cite{laur08-1}. 

The rational Lie algebra $\n^\Q$ is defined as the vector space over $\Q$ with basis $$\{ X_1, X_2, X_3, X_4, Y_1, Y_2, Z_1, Z_2, V_1, V_2, W_1, W_2 \}$$ and Lie bracket given by the relations
\begin{alignat*}{5}
    [X_1, X_3] &= Y_1 \quad \quad& [X_1, Y_1] &= Z_1 \quad \quad& [X_1, Z_1] &= V_1 \quad \quad& [X_3, V_1] &= W_1\\
    [X_2, X_4] &= Y_2 \quad \quad& [X_2, Y_2] &= Z_2 \quad \quad& [X_2, Z_2] &= V_2 \quad \quad& [X_4, V_2] &= W_2\\
    &&&&\\
    [Z_1, Y_1] &= W_1 \quad \quad & [X_1, X_4] &= W_1& &\\
    [Z_2, Y_2] &= W_2 \quad \quad & [X_2, X_3] &= W_2.& &
\end{alignat*}
The last two brackets ensure that $\n^\Q$ is not isomorphic to the direct sum of two filiform Lie algebras, where a filiform Lie algebra is a nilpotent Lie algebra of nilpotency class $c > 1$ and dimension $c+1$, being the smallest possible dimension of such a Lie algebra. 

In order to apply Theorem \ref{thm:constructingAnosov}, we take $E$ the field $\Q(\sqrt{k}) \subset \R$ where $k$ is a positive integer which is not a square. This is a quadratic extension of $\Q$ and its Galois group is given by $\gal(E, \Q) = \{\text{Id}, \tau\}$ with $\tau(a + b\sqrt{k}) = a - b \sqrt{k}$ for all $a, b \in \Q$. Let $\xi \in E$ be any algebraic integer with minimal polynomial of degree $2$, which exists by Dirichlet's unit theorem. It has only one conjugate which is also its inverse $\xi^{-1} = \tau(\xi)$, moreover it holds that $\vert \xi \vert \neq 1 \neq \vert \xi^{-1}\vert$. We define the linear map $A:\n^E \to \n^E$ and the representation $\rho:\{1, \tau\} \to \Aut(\n^\Q)$ by
\begin{equation*}
    A = \begin{pmatrix}
    B^{3} & & & & & \\
    & B^{-2} & & & & \\
    & & B & & & \\
    & & & B^{4} & & \\
    & & & & B^{7} & \\
    & & & & & B^{5}
    \end{pmatrix},
    \quad \quad \quad
    \rho_{\tau} = \begin{pmatrix}
    C & & & & & \\
    & C & & & & \\
    & & C & & & \\
    & & & C & & \\
    & & & & C & \\
    & & & & & C
    \end{pmatrix}
\end{equation*}
with respect to the basis $\{ X_1, X_2, X_3, X_4, Y_1, Y_2, Z_1, Z_2, V_1, V_2, W_1, W_2 \}$ and where
\begin{equation*}
    B = \begin{pmatrix}
    \xi &0\\
    0& \xi^{-1}
    \end{pmatrix}, \quad \quad \quad 
    C = \begin{pmatrix}
    0& 1\\
    1 &0
    \end{pmatrix}.
\end{equation*}
One can check that $A$ and $\rho_\tau$ are automorphisms of the Lie algebra $\n^E$. Since
\begin{align*}
    C B C^{-1} = \begin{pmatrix} \xi^{-1}& 0\\
    0 & \xi
    \end{pmatrix} = \begin{pmatrix} \tau(\xi)& 0\\
    0 & \tau(\xi^{-1})
    \end{pmatrix} = B^\tau
\end{align*}
it follows that $\rho_{\tau} A \rho_{\tau^{-1}} = A^\tau$.  Therefore all prerequisites of Theorem \ref{thm:constructingAnosov} are satisfied and $A$ induces an Anosov automorphism on the rational form $\m^\Q_k \subset \n^E$. We give $\m_k^\Q$ the subscript $k$ since its structure depends on the choice of field extension $E = \Q(\sqrt{k})$. 

To see that the Lie algebra $\n^\Q$ has no positive grading, or equivalently it has no expanding automorphism, we use \cite[Theorem 4.4.]{dere14-1}. Assume by contradiction that there would be an expanding automorphism, then this result implies that there exists an expanding automorphism $\varphi \in \Aut(\n^\Q)$ which commutes with the map $A$ above. Since all the eigenvalues of $A$ are distinct, this implies the expanding automorphism is diagonal in the basis $X_1, \ldots, W_2$, in particular we have $\varphi(X_i) = \lambda_i X_i$ with $\lambda_i > 1$ for $1 \leq i \leq 4$. The relations on $\n^\Q$ imply that 
\begin{alignat*}{5}
  \varphi(Y_1) &= \lambda_1 \lambda_3 Y_1 \quad \quad & \varphi(Z_1) &= \lambda_1^2 \lambda_3 Z_1 \quad \quad  &\varphi(V_1) &= \lambda_1^3 \lambda_3 V_1 \quad \quad  &\varphi(W_1) &= \lambda_1^3 \lambda_3^2 W_1\\
\varphi(Y_2) &= \lambda_2 \lambda_4 Y_2 \quad \quad & \varphi(Z_2) &= \lambda_2^2 \lambda_4 Z_2 \quad \quad & \varphi(V_2) &= \lambda_2^3 \lambda_4 V_2 \quad \quad & \varphi(W_2) &= \lambda_2^3 \lambda_4^2 W_2.
\end{alignat*} The last relations moreover imply that $\lambda_1 \lambda_4 = \lambda_1^3  \lambda_3^2 $ and $\lambda_2 \lambda_3 = \lambda_2^3 \lambda_4^2$. Combining these equations, we get that $\lambda_4 = (\lambda_1 \lambda_3)^2 = \lambda_1^2 \lambda_2^4 \lambda_4^4$ or thus $1 = \lambda_1^2 \lambda_2^4 \lambda_4^3$ which is a contradiction since $\varphi$ is expanding.

By using the explicit form of the rational Lie algebra $\m^\Q_k$, it is possible to compute the bracket relations. We first need a basis for which we can take
\begin{alignat*}{3}
    \overline{X_1} &= X_1 + X_2 \quad \quad\quad& \overline{Y_1} &= Y_1 + Y_2\quad \quad\quad & \overline{V_1} &= V_1 + V_2\\
    \overline{X_2} &= \sqrt{k}(X_1 - X_2) \quad \quad\quad& \overline{Y_2} &= \sqrt{k}(Y_1 - Y_2) \quad \quad\quad& \overline{V_2} &= \sqrt{k}(V_1 - V_2)\\
    \overline{X_3} &= X_3 + X_4 \quad \quad\quad& \overline{Z_1} &= Z_1 + Z_2 \quad\quad\quad & \overline{W_1} &= W_1 + W_2\\
    \overline{X_4} &= \sqrt{k}(X_3 - X_4) \quad\quad \quad& \overline{Z_2} &= \sqrt{k}(Z_1 - Z_2) \quad\quad \quad& \overline{W_2} &= \sqrt{k}(W_1 - W_2).
\end{alignat*}
For example we have $\rho_\tau(\overline{X_2}) = \rho_\tau(\sqrt{k}(X_1 - X_2)) = \sqrt{k}(X_2 - X_1) = -\sqrt{k}(X_1 - X_2) = \tau(\sqrt{k})(X_1 - X_2) = \overline{X_2}^\tau$ which shows that indeed $\overline{X_2}$ lies in $\m_k^\Q$.
As one can now calculate, the Lie bracket on $\m^\Q_k$ is given by the relations
\begin{alignat*}{3}
    [\overline{X_1}, \overline{X_3}] &= \overline{Y_1} + \overline{W_1} \quad \quad &[\overline{X_1}, \overline{Y_1}] &= \overline{Z_1}\quad \quad &[\overline{X_1}, \overline{Z_1}] &= \overline{V_1}\\
    [\overline{X_2}, \overline{X_4}] &= k( \overline{Y_1} - \overline{W_1}) \quad \quad &[\overline{X_2}, \overline{Y_2}] &= k\overline{Z_1}\quad \quad &[\overline{X_2}, \overline{Z_2}] &= k\overline{V_1}\\
    [\overline{X_1}, \overline{X_4}] &=  \overline{Y_2} - \overline{W_2}& \quad \quad [\overline{X_1}, \overline{Y_2}] &= \overline{Z_2}\quad \quad& [\overline{X_1}, \overline{Z_2}] &= \overline{V_2}\\
    [\overline{X_2}, \overline{X_3}] &= \overline{Y_2} + \overline{W_2}& \quad \quad [\overline{X_2}, \overline{Y_1}] &= \overline{Z_2}& \quad \quad [\overline{X_2}, \overline{Z_1}] &= \overline{V_2}\\
    \\
    [\overline{X_3}, \overline{V_1}] &= \overline{W_1}\quad \quad &[\overline{Z_1}, \overline{Y_1}] &= \overline{W_1}\\
    [\overline{X_4}, \overline{V_2}] &= k\overline{W_1}\quad \quad &[\overline{Z_2}, \overline{Y_2}] &= k\overline{W_1}\\
    [\overline{X_3}, \overline{V_2}] &= \overline{W_2}\quad \quad& [\overline{Z_1}, \overline{Y_2}] &= \overline{W_2}\\
    [\overline{X_4}, \overline{V_1}] &= \overline{W_2}& \quad \quad [\overline{Z_2}, \overline{Y_1}] &= \overline{W_2}.
\end{alignat*}
Clearly we have that $\faktor{\m_k^\Q}{[\m_k^\Q,[\m_k^\Q, \m_k^\Q]]}$ is a rational Lie algebra of type $(4, 2)$. In \cite{laur08-1} a complete list of rational Lie algebras of type $(4, 2)$ up to isomorphism is given. From this it follows that if $\frac{k}{l}$ is not the square of an integer, then the rational Lie algebras $\m_k^\Q$ and $\m_l^\Q$ are non-isomorphic. On the other hand, if $k = l m^2$ for some integer $m \in \Z$, then the fields $\Q(\sqrt{k}) = \Q(\sqrt{l})$ are equal and hence the Lie algebras $\m_k^\Q= \m_l^\Q$ by the construction above.

We conclude that the Lie algebras $\m_k^\Q$ have the claimed properties. In the following section we will use this family to give a family of Anosov Lie algebras without expanding automorphism in every dimension $\geq 14$.

\section{Indecomposable factors of Anosov Lie algebras}
\label{sec:decom}

Every Lie algebra has an essentially unique decomposition into indecomposable factors. In this section, we demonstrate that the existence of an Anosov automorphism is equivalent to the existence of an Anosov automorphism on every indecomposable factor, and the identical question for the existence of a positive grading. The special case of abelian factors and Anosov automorphisms was considered in \cite{lw08-1}. In particular, the results of this section allow us to construct families of Anosov Lie algebras with no expanding maps in every dimension $\geq 14$. 

Recall that a Lie algebra $\n^E$ is called indecomposable if it is not equal to the direct sum of two non-trivial ideals, i.e.~if $\n^E= \g \oplus \h$ with $\g$ and $\h$ ideals of $\n^E$, then either $\g = 0$ or $\h=0$. Since we only consider finite-dimensional Lie algebras in this paper, every Lie algebra $\n^E$ has a decomposition $\n^E= \g_1 \oplus \ldots \oplus \g_k$, with the ideals $\g_i$ indecomposable. In this case we call $\g_i$ a factor of $\n^E$. From \cite{fgh13-1} it follows that such a decomposition is unique up to isomorphism and up to reordering the factors $\g_i$. Note that being indecomposable depends on the field over which you are working.

The first step we need is a description of the automorphisms on a direct sum of indecomposable Lie algebras. The following result was proven in \cite[Theorem 3.4]{fgh13-1} for real Lie algebras, but since the proof does not depend on the field over which the Lie algebra is defined, we can state it here for a general Lie algebra $\n^E$ defined over a field $E$ which extends the rationals $\Q$.

\begin{theorem}
\label{thm:autoDirectSum}
    Let $\n^E$ be a Lie algebra over $E$ with a decomposition $\n^E = \g_1 \oplus \ldots \oplus \g_k$ into indecomposable factors. A bijective linear map $\phi:\n^E \to \n^E$ is an automorphism of $\n^E$ if and only if it has the form $\theta + \eta$, where $\theta$ is an automorphism of $\n^E$ which maps each $\g_i$ to itself or to an isomorphic summand $\g_j$, and where $\eta$ is a linear map of $\n^E$ such that $\eta(\n^E) \subset Z(\n^E)$ and $\eta([\n^E, \n^E]) = \{0\}$.
\end{theorem}

To prove the main theorems, we will apply this result to an Anosov automorphism or an expanding automorphism. Since the properties we consider are invariant under taking non-zero powers, the following lemma will be useful.

\begin{lemma}
    \label{lem:powerAutoDirectSum}
    Let $\n^E = \g_1 \oplus \ldots \oplus \g_k$ be a direct sum of indecomposable Lie algebras and $\phi$ an automorphism of $\n^E$. There exists an integer $l > 0$ such that $\phi^l$ has a decomposition $\phi^l = \theta_l + \eta_l$ satisfying the properties of Theorem \ref{thm:autoDirectSum} and with $\theta_l(\g_i) = \g_i$ for all $1 \leq i \leq k$.
\end{lemma}
\begin{proof}
Given the decomposition $\phi =\theta + \eta$, we give a decomposition for the map $\phi^m$. Let us first define for all positive integers $m > 0$ the linear maps $\eta_m = \phi^m - \theta^m$. 

We start by proving inductively that $\eta_m(\n^E) \subset Z(\n^E)$ and $\eta_m([\n^E, \n^E]) = \{0\}$ for any $m > 0$. For $m = 1$ we have $\eta_1 = \eta$ and thus by the assumption it holds that $\eta_1(\n^E) \subset Z(\n^E)$ and $\eta_1([\n^E, \n^E]) = \{0\}$. Now assume that $\eta_m(\n^E) \subset Z(\n^E)$ and $\eta_m([\n^E, \n^E]) = \{0\}$ for a fixed $m > 0$. We can rewrite $\eta_{m+1}$ as
    \begin{equation*}
        \eta_{m +1} = \phi^{m+1} - \theta^{m +1} = (\theta + \eta)\phi^{m} - \theta \theta^m = \theta \eta_m + \eta \phi^m.
    \end{equation*}
    Since $\theta$ and $\phi^m$ are automorphisms, they preserve both the center $Z(\n^E)$ and the derived algebra $[\n^E, \n^E]$. This fact together with the induction hypothesis clearly gives $\eta_{m+1}(\n^E) \subset Z(\n^E)$ and $\eta_{m+1}([\n^E, \n^E]) = \{0\}$. 
    
    Note that the automorphism $\theta$ induces a permutation $\sigma \in S_k$ such that $\theta(\g_i) = \g_{\sigma(i)}$ for all $1 \leq i \leq k$. Let $l$ be the order of this permutation $\sigma$. It follows that $\theta^l(\g_i) = \g_i$ for all $1 \leq i \leq k$. We can decompose the automorphism $\phi^l$ as $\phi^l = \theta^l + (\phi^l - \theta^l) = \theta^l + \eta_l$ which now satisfies all properties of the above theorem for the automorphism $\phi^l$. If we set $\theta_l = \theta^l$, then $\theta_l(\g_i) = \g_i$ for all $1 \leq i \leq k$, which proves the lemma.
\end{proof}

Note that when studying decompositions of Lie algebras, often abelian factors behave differently from other factors of the Lie algebra, see \cite[Theorem 3.1.]{dere20-1}. 
It is immediate that a subspace $\mathfrak{a} \subset \n^E$ is an abelian factor if and only if it is contained in the center $Z(\n^E)$ with $\mathfrak{a} \cap [\n^E, \n^E] = 0$. Therefore we define the number $m(\n^E) = \dim Z(\n^E) -  \dim Z(\n^E) \cap [\n^E, \n^E]$, which is exactly the dimension of any maximal abelian factor. Note that this number does not depend on the field, i.e. if $E \subset F \subset \C$ is an extension of fields, we have that $m(\n^E) = m(\n^F = F \otimes_E \n^E)$. In \cite[Theorem 3.1]{lw08-1} it has been proved that whenever a rational Anosov Lie algebra $\n^\Q = \tilde{\n} \oplus \mathfrak{a}$ has a non-trivial maximal abelian factor $\mathfrak{a}$, both $\tilde{\n}$ and $\mathfrak{a}$ are Anosov, where the last condition is equvalent to $m(\n^\Q) \geq 2$. We generalize this result by using the following two lemma's, with their proofs following the idea from \cite{lw08-1}.

\begin{lemma}
    \label{lem:blockDiagHypIntLike}
    If $M \in \GL_n(\Q)$ is an integer-like hyperbolic matrix which is block-diagonal:
    \begin{align*}
        M = \begin{pmatrix} M_1&&\\
        &\ddots&\\
        &&M_k
        \end{pmatrix},
    \end{align*}
    then each $M_i$ is integer-like and hyperbolic as well.
\end{lemma}
\begin{proof}
    Denote by $f, f_1, \ldots, f_k$ the characteristic polynomials of $M, M_1, \ldots, M_k$, respectively, then we have that $f = \displaystyle \prod_{i=1}^k f_i$. This immediately implies that if the $M_i$ are integer-like and hyperbolic, then also $M$ is integer-like and hyperbolic. 

    For the other direction, we know that the roots of $f$ are hyperbolic algebraic units and therefore the roots of each polynomial $f_i$ are hyperbolic algebraic units as well. Using that the algebraic integers form a ring, we find that the coefficients of the polynomials $f_i$ are algebraic integers. Since these coefficients also have to be rational, we must conclude that each $f_i$ is an integer polynomial. The constant term of each $f_i$ is a product of algebraic units and thus an algebraic unit itself. Since they also have to lie in $\Z$ we get that the constant term of each $f_i$ is equal to $\pm 1$. This shows that each matrix $M_i$ is hyperbolic and integer-like.
\end{proof}

\begin{lemma}
    \label{lem:AutMaxAbelianFactorAndComplement}
    Let $\n^E$ be a Lie algebra with a decomposition $\n^E = \mathfrak{a} \oplus \g$ where $\mathfrak{a}$ is a maximal abelian factor of $\n^E$. Then for any semi-simple automorphism $\phi \in \Aut(\n^E)$ there exists an automorphism $\varphi \in \Aut(\n^E)$ such that $(\varphi \phi \varphi^{-1})(\mathfrak{a}) = \mathfrak{a}$ and $(\varphi \phi \varphi^{-1})(\g) = \g$.
\end{lemma}

\begin{proof}
    Let us write $\mathfrak{b} = Z(\n^E) \cap [\n^E, \n^E]$. Since $\phi$ is an automorphism, it must preserve the center and the derived algebra and therefore also $\mathfrak{b}$. Using that $\phi$ is semi-simple, there exists a complementary subspace $\tilde{\mathfrak{a}}$ to $\mathfrak{b}$ in $Z(\n^E)$ for which $\phi(\tilde{\mathfrak{a}}) = \tilde{\mathfrak{a}}$. Clearly $\tilde{\mathfrak{a}}$ is a maximal abelian factor of $\n^E$. Since $m(\g) = 0$, it follows that $\tilde{\mathfrak{a}} \cap \g = \{0\}$ and thus that $\n^E = \tilde{\mathfrak{a}} \oplus \g$. Using again that $\phi$ is semi-simple, we get a subspace $\h$ which is a complementary subspace to $Z(\n^E)$ with $\phi(\h) = \h$. This gives the $\phi$-invariant decomposition
    \begin{equation*}
        \n^E = \tilde{\mathfrak{a}} \oplus \underbrace{\mathfrak{b} \oplus \h}_{:= \tilde{\g}}.
    \end{equation*}
    An easy check shows that the subspace $\tilde{\mathfrak{g}}$ is an ideal of $\n^E$. We now clearly get the isomorphisms of Lie algebras $\tilde{\mathfrak{a}} \approx \mathfrak{a}$ and $\tilde{\g} \approx (\tilde{\mathfrak{a}} \oplus \tilde{\g})/\tilde{\mathfrak{a}} \approx (\tilde{\mathfrak{a}} \oplus \g)/\tilde{\mathfrak{a}} \approx \g$. By applying these isomorphisms component-wise we get a map $\varphi: \tilde{\mathfrak{a}} \oplus \tilde{\g} \to \mathfrak{a} \oplus \mathfrak{g}$ which is the automorphism on $\n^E$ that we want.
\end{proof}

Now we are ready to prove the main theorems of this section.

\begin{theorem}
    \label{thm:AutomorphismsOnDirectSumOfIndecomp}
    Let $\n^E$ be a Lie algebra with decomposition
    \begin{equation*}
        \n^E = \mathfrak{a} \oplus \g_1 \oplus \ldots \oplus \g_k
    \end{equation*}
    into factors with $\g_i$ non-abelian and indecomposable and $\mathfrak{a}$ abelian. Then for any automorphism $\phi \in \Aut(\n^E)$ there exists an integer $l > 0$ and an automorphism $\tilde{\phi} \in \Aut(\n^E)$ which has the same characteristic polynomial as $\phi^l$ and such that $\tilde{\phi}(\mathfrak{a}) = \mathfrak{a}$ and $\tilde{\phi}(\g_i) = \g_i$ for all $1 \leq i \leq k$.
\end{theorem}

\begin{proof}
    Since $\Aut(\n^E)$ is an linear algebraic group, the semi-simple part of $\phi$ is also an automorphism of $\n^E$. It has the same characteristic polynomial as $\phi$ and thus we can assume that $\phi$ is semi-simple to start with. 
    
    Write $\g = \g_1 \oplus \ldots \oplus \g_k$. It is clear that $\mathfrak{a}$ is a maximal abelian factor of $\n^E$. From Lemma \ref{lem:AutMaxAbelianFactorAndComplement} we get an automorphism $\varphi \in \Aut(\n^E)$ such that $(\varphi \phi \varphi^{-1})(\mathfrak{a}) = \mathfrak{a}$ and $(\varphi \phi \varphi^{-1})(\g) = \g$. Therefore we get the well-defined restrictions $\phi_1 := (\varphi \phi \varphi^{-1})|_\mathfrak{a}$ and $\phi_2 := (\varphi \phi \varphi^{-1})|_{\g}$. Lemma \ref{lem:powerAutoDirectSum} then gives us an integer $l > 0$ such that $\phi_2^l = \theta + \eta$ with $\theta$ an automorphism on $\g$ such that $\theta(\g_i) = \g_i$ for all $1 \leq i \leq k$ and $\eta$ a linear map on $\g$ such that $\eta(\g) \subset Z(\g)$ and $\eta([\g, \g]) = \{0\}$. Since $\mathfrak{a}$ is a maximal abelian factor of $\n^E$ we must have $m(\g) = 0$. By consequence $Z(\g) \subset [\g, \g]$ and the map $\eta$ now also satisfies $\eta(\g) \subset [\g, \g]$. Write $\h$ for some complement of $[\g, \g]$ in $\g$, the matrix representation of the maps $\theta$ and $\eta$ with respect to the direct sum $\g = [\g, \g] \oplus \h$ takes the form
    \begin{equation*}
        \theta = \begin{pmatrix}
        \theta_{11} & \theta_{12}\\
        0 & \theta_{22}
        \end{pmatrix}, \quad \quad \eta = \begin{pmatrix}
        0 & \eta_{12}\\
        0 & 0\end{pmatrix}.
    \end{equation*}
    From this it follows that $\theta$ has the same characteristic polynomial as $\phi_2^l$. By consequence the automorphism $\tilde{\phi}$ defined by $\tilde{\phi}|_{\mathfrak{a}} = \phi_1^l$ and $\tilde{\phi}|_\g = \theta$ has the same characteristic polynomial as $\phi^l$ and preserves every factor of the decomposition $\n^E = \mathfrak{a} \oplus \g_1 \oplus \ldots \oplus \g_k$. This proves the theorem.
\end{proof}

Combining Theorem \ref{thm:AutomorphismsOnDirectSumOfIndecomp} with Lemma \ref{lem:blockDiagHypIntLike} gives us the result on the existence of Anosov automorphisms on a direct sum of indecomposable factors.

\begin{theorem}
 \label{thm:directSumAnosov}
Let $\n^\Q$ be a rational Lie algebra with a decomposition $$\n^\Q = \mathfrak{a} \oplus \mathfrak{g}_1 \oplus \ldots \oplus \mathfrak{g}_k$$ into factors with $\mathfrak{g}_i$ non-abelian and indecomposable and $\mathfrak{a}$ abelian. The Lie algebra $\n^\Q$ is Anosov if and only if $\dim(\mathfrak{a}) \geq 2$ and every factor $\mathfrak{g}_i$ is Anosov. 
\end{theorem}
\begin{proof}
    Note that $\mathfrak{a}$ is Anosov if and only if $\dim(\mathfrak{a}) \geq 2$. Write $\mathfrak{g}_0$ for the abelian Lie algebra $\mathfrak{a}$. If $\mathfrak{n}^\Q = \displaystyle \bigoplus_{i=0}^k \g_i$ and every $\mathfrak{g}_i$ is Anosov with Anosov automorphism $A_i$ respectively, then the automorphism $A$ defined by $A|_{\g_i} = A_i$ is Anosov on $\n$. 
    
    Conversely, assume that $\n$ is Anosov with Anosov automorphism $A$. Let $\tilde{A}$ be the automorphism we get from Theorem \ref{thm:AutomorphismsOnDirectSumOfIndecomp}. It is clear that $\tilde{A}$ is also integer-like hyperbolic and thus an Anosov automorphism on $\n^\Q$ with $\tilde{A}(\g_i) = \g_i$ for all $0 \leq i \leq k$. Using Lemma \ref{lem:blockDiagHypIntLike} we get that $\tilde{A}|_{\g_i}$ defines an Anosov automorphism on $\g_i$ for all $0 \leq i \leq k$. In particular we get that $\dim(\mathfrak{a}) = \dim(\g_0) \geq 2$.
\end{proof}

A similar result is obtained for expanding automorphisms.

\begin{theorem}
 \label{thm:directSumExp}
Let $\n^E$ be a Lie algebra with a decomposition $$\n^E = \mathfrak{a} \oplus \mathfrak{g}_1 \oplus \ldots \oplus \mathfrak{g}_k$$ into factors with $\mathfrak{g}_i$ non-abelian indecomposable and $\mathfrak{a}$ abelian. The Lie algebra $\n^E$ has a positive grading if and only if every factor $\mathfrak{g}_i$ has a positive grading.
\end{theorem}
\begin{proof}
    Note that the existence of a positive grading on $\n^E$ is equivalent to the existence of an expanding automorphism on $\n^E$ and that an abelian Lie algebra always admits a positive grading, namely the trivial one. Assume each factor $\g_i$ admits an expanding automorphism $\phi_i$. Write $\g_0$ for the abelian Lie algebra $\mathfrak{a}$ and let $\phi_0$ be an expanding automorphism on $\g_0$. The automorphism $\phi$ on $\n^E$ defined by $\phi|_{\g_i} = \phi_i$ for all $0 \leq i \leq k$ is an expanding automorphism on $\n^E$.
    
    Conversely, assume that $\n^E$ admits an expanding automorphism $\phi$. Theorem \ref{thm:AutomorphismsOnDirectSumOfIndecomp} gives us an integer $l > 0$ and an automorphism $\tilde{\phi}$ with the same characteristic polynomial as $\phi^l$ and such that $\tilde{\phi}(\g_i) = \g_i$ for all $0\leq i \leq k$. It is clear that each $\tilde{\phi}|_{\g_i}$ is an expanding automorphism as well and thus that each $\mathfrak{g}_i$ has a positive grading.
\end{proof}

\begin{cor}
For every dimension $n \geq 14$, there exists an Anosov Lie algebra $\n^\Q$ with no positive grading.
\end{cor}

\begin{proof}
This follows immediately by considering the Lie algebra $\m_k^\Q \oplus \Q^{n - 12}$, which satisfies the properties by Theorem \ref{thm:directSumAnosov} and \ref{thm:directSumExp}.
\end{proof}

As another consequence, these results imply that an example of minimal dimension must be indecomposable, reducing the possibilities in the following sections.

\section{Rank of Anosov automorphisms}
\label{sec:rank}


	By definition, the eigenvalues of an Anosov automorphism are algebraic units, with certain relations induced by the fact that they are eigenvalues of an automorphisms of the Lie algebra. Therefore it is important to study the multiplicative group generated by these eigenvalues. We will define the rank of an Anosov automorphism as the rank of this abelian group, i.e. the maximal number of $\mathbb{Z}$-independent elements. Using the rank will prove useful for constructing positive gradings on Lie algebras. We start by introducing some terminology, partially coming from \cite{payn09-1}.
	
	Let us call a monic polynomial $f \in \Z[X]$ an \textit{Anosov polynomial} if it has only roots of absolute value different from 1 and constant term equal to $\pm 1$. The Anosov polynomials are the exactly the characteristic polynomials of Anosov automorphisms. Note that there are no Anosov polynomials of degree one and that by Lemma \ref{lem:blockDiagHypIntLike} the irreducible factors of an Anosov polynomial over $\Q$ are Anosov as well. Let us fix an Anosov polynomial $f$ of degree $k$ with roots $\lambda_1, \ldots, \lambda_k$. We define the rank of $f$ to be the rank of the abelian multiplicative group generated by the roots $\lambda_1, \ldots, \lambda_k$. This group can also be seen as the image of the following map
	\begin{align}
	    \label{eq:mapPhi}
	    \phi_f: \Z^k \to E^\times: (z_1, \ldots, z_k) \mapsto \lambda_1^{z_1} \cdot \ldots \cdot \lambda_k^{z_n}
	\end{align}
	where $E$ is the splitting field of $f$ over $\Q$. By consequence the rank of $f$ can be expressed as $k - \rank(\ker \phi_f)$ and thus lies between $0$ and $k$. 
	
	Since we assumed the constant term of $f$ to be $\pm1$, the product of all its roots is equal to $\pm 1$. This implies that $\Z (2, \ldots, 2) \in \ker \phi_f$ and thus that the rank of an Anosov polynomial $f$ is at most $k - 1$. When there is equality, we say that $f$ satisfies the \textit{full rank condition} or that its roots have \textit{full rank}. Note that every polynomial that satisfies the full rank condition must be irreducible. This notion was first introduced in \cite{payn09-1}. Combining \cite[Proposition 3.6.(2)]{payn09-1} with Lemma \ref{lem:IrrPrimeDegreeCyclicPerm} we obtain the following result.
	
	\begin{prop}		\label{prop:primeDegFullRank}
		Let $f$ be an irreducible Anosov polynomial of prime degree, then the roots of $f$ satisfy the full rank condition.
	\end{prop}
	
	\noindent Note that there is also a lower bound on the rank of $f$, namely it has to be at least 1. If the rank of $f$ would be $0$, then the group generated by the roots would be finite which implies that each $\lambda_i$ is a root of unity. This contradicts the fact that their absolute value has to be different from 1.
	
	We can naturally extend the notion of rank to Anosov automorphisms by using their characteristic polynomial.
	\begin{definition}
	    Let $A$ be an Anosov automorphism on a rational Lie algebra $\n^\Q$. We define the \textit{rank} of $A$ to be the rank of the multiplicative group generated by its eigenvalues.
	\end{definition}

	Let $\n^\Q$ be a nilpotent Anosov Lie algebra with Anosov automorphism $A$. We denote by $f$ the characteristic polynomial of $A$ and with $E$ the splitting field of $f$ over $\Q$. Since $\Aut(\n^\Q)$ is a linear algebraic group, it contains the semi-simple and nilpotent part of its elements, and thus we can assume $A$ to be semi-simple, see also \cite[Proposition 2.2.]{lw08-1}. This gives a decomposition
	\begin{align*}
	    \n^\Q = \n_1 \oplus \ldots \oplus \n_c
	\end{align*}
	of $\n^\Q$ as a vector space such that $A (\n_i) = \n_i$ for all $1 \leq i \leq c$ and the $i$-th ideal of the lower central series of $\n^\Q$, denoted by $\gamma_i(\n^\Q)$, is equal to $\gamma_i({\n^\Q}) = \n_i \oplus \ldots \oplus \n_c$. Here $c$ denotes the nilpotency class of $\n^\Q$. We write $n_i$ for the dimension of $\n_i$ and define the \textit{type} of $\n^\Q$ to be the $c$-tuple $(n_1, \ldots, n_c)$. 
	
	Let us denote by $A_i:\n_i \to \n_i$ the restriction of $A$ to $\n_i$. It is clear from Lemma \ref{lem:blockDiagHypIntLike} that each $A_i$ is also hyperbolic and integer-like. Let $f_i$ denote the characteristic polynomial of $A_i$ for all $1 \leq i \leq c$. Since $A$ is semi-simple, $\n^E$ has a basis consisting of eigenvectors of $A$ which respect the decomposition $\n^E = \n_1^E \oplus \ldots \oplus \n_c^E$. Let $X_1, \ldots, X_{n_1}$ be a basis of eigenvectors for $A_1$. Since the $i$-fold brackets of $X_1, \ldots, X_{n_1}$ span the subspace $\n_i^E$, it follows from the fact that $A$ is an automorphism that the eigenvalues of $A_i$ are $i$-fold products of the eigenvalues of $A_1$. Therefore the rank of $A$ is actually equal to the rank of $f_1$. For the Anosov automorphism $A$ we will write $\phi_A := \phi_{f_1}$ as defined by equation (\ref{eq:mapPhi}). The Anosov automorphism $A$ is said to be of \textit{full rank} if $f_1$ satisfies the full rank condition.

	\begin{lemma}
	The rank of $A$ is equal to the rank of $A^k$ for every $k > 0$.
	\end{lemma}
	\begin{proof}
		Consider the two morphisms $\phi_A: \Z^{n_1} \to E$ and $\phi_{A^k}: \Z^{n_1} \to E'$ as defined above where $E$ and $E'$ are the splitting fields of the characteristic polynomials of $A$ and $A^k$, respectively. It is clear that $E' \subset E$ since the eigenvalues of $A^k$ are $k$-powers of the eigenvalues of $A$. Let $\theta_k$ be the morphism from $\Z^{n_1}$ to itself given by multiplication by $k$, then we have $\phi_{A^k} = \phi_A \circ \theta_k$. Since $\theta_k$ is an injective morphism, it preserves $\Z$-linear independence. Therefore it follows that $\rank (\ker \phi_{A^k}) \leq \rank (\ker \phi_A)$. On the other hand we also have that $\ker \phi_A \subset \ker \phi_{A^k}$ which implies that $\rank (\ker \phi_A) = \rank (\ker \phi_{A^k})$. We conclude that the rank of an Anosov automorphism is invariant under taking non-zero powers.
	\end{proof}
	\noindent During the remainder of this paper, we will often take powers of Anosov automorphisms to achieve stronger assumptions. For example, we will from now on always assume that the constant term of $f_1$ is equal to $1$, by squaring $A$ if necessary.
	
	Consider the following morphism between abelian groups
    	\begin{align*}
	    \psi: \Z^{n_1} \to \Z: (z_1, \ldots, z_{n_1}) \mapsto \sum_{i = 1}^{n_1} z_i.
	\end{align*}
	For any Anosov automorphism $A$, the image of $\ker \phi_A$ under this map is a subgroup of $\Z$. Therefore there exists a unique positive integer $d_A$ such that $\psi (\ker \phi_A) = d_A \cdot \Z$. In the special case where $A$ has full rank, $d_A$ equals $n_1$. The following theorem gives a restriction on the Lie algebra structure using this integer $d_A$.
	
	\begin{prop}
		\label{prop:CnotTooHighThenGraded}
		Let $A$ be an Anosov automorphism on the rational nilpotent Lie algebra $\n^\Q$ with corresponding decomposition $\n^\Q = \n_1 \oplus \ldots \oplus \n_c$. Then for all $1 \leq i,j \leq c$, we have that
		\begin{align*}
			[\n_i, \n_j] \subset \bigoplus_{k \in \N} \n_{i + j + k\cdot d_A}
		\end{align*}
		where we set $\n_i = 0$ if $i > c$. By consequence we have that if $c \leq d_A + 1$, the Lie algebra $\n^\Q$ is positively graded.
	\end{prop}
	
	\begin{proof}
		Let $A:\n^\Q \to \n^\Q$ be the Anosov automorphism for which $A (\n_i) = \n_i$. We denote the splitting field of the characteristic polynomial of $A_1$ by $E$ and the roots by $\lambda_1, \ldots, \lambda_{n_1}$. Then $A$ has a basis of eigenvectors in $\n^E = E \otimes \n^\Q$ and since the subspaces $\n_i^E := E \otimes \n_i$ are invariant under $A$, it follows that we can choose a basis of eigenvectors that respects the direct sum $\n^E = \n_1^E \oplus \ldots \oplus \n_c^E$. 
		
		Now take any two of these eigenvectors $X, Y$ with $X \in \n_k^E$ and $Y \in \n_l^E$. If their bracket $[X,Y]$ is non-zero, it is again an eigenvector of $A$. Define $\mathcal{I}$ as the set of all integers $m$ with $1 \leq m \leq c$ such that the eigenvalue of $[X, Y]$ is an eigenvalue of $A_m$. Now take an $m \in \mathcal{I}$. Since the eigenvalues of $A_i$ are $i$-fold products of the $\lambda_j$'s, there exist non-negative integers $e_1, \ldots, e_{n_1}, f_1, \ldots, f_{n_1}, g_{1}, \ldots, g_{n_1}$ with $\displaystyle \sum_{i = 1}^{n_1} e_i = k, \sum_{i = 1}^{n_1} f_i = l$ and $\displaystyle \sum_{i = 1}^{n_1} g_i = m$ such that the eigenvalues of $X, Y$ and $[X,Y]$ are given by $\displaystyle \prod_{i = 1}^{n_1} \lambda_i^{e_i}, \prod_{i = 1}^{n_1} \lambda_i^{f_i}$ and $\displaystyle \prod_{i = 1}^{n_1} \lambda_i^{g_i}$, respectively. Since $A$ is an automorphism of $\n^\Q$, we must have that the product of the eigenvalues of $X$ and $Y$ is equal to the eigenvalue of $[X,Y]$. This implies that
		\begin{align*}
			&\left(\prod_{i = 1}^{n_1} \lambda_i^{e_i} \right) \cdot \left(  \prod_{i = 1}^{n_1} \lambda_i^{f_i}\right) =  \prod_{i = 1}^{n_1} \lambda_i^{g_i}
			\quad \quad \Rightarrow \quad \quad \prod_{i = 1}^{n_1} \lambda_i^{e_i + f_i - g_i} = 1.
		\end{align*}
	By consequence $(e_1 + f_1 - g_1, \ldots, e_{n_1} + f_{n_1} - g_{n_1}) \in \ker \phi_A$ and so we know that its image under $\psi$ is a multiple of $d_A$. This gives that $m = k + l + r\cdot d_A$ for some $r \in \Z$. Since $m$ was chosen arbitrarily in $\mathcal{I}$, we know that $\mathcal{I} \subset \{ k + l + r\cdot d_A \mid r \in \Z \}$. Therefore we must have that $[X, Y]$ lies in the direct sum $\displaystyle \bigoplus_{r \in \Z} \n^E_{k + l + r\cdot d_A}$, where we set $\n_i^E = 0$ for $i < 1$ or $i > c$. As the eigenvectors $X \in \n^E_k$ and $Y \in \n_l^E$ were arbitrary, it follows that $\displaystyle [\n_k^E, \n_l^E] \subset \bigoplus_{r \in \Z} \n_{k + l + r\cdot d_A}^E$. From this we get that the same holds for the rational spaces $\n_i$ and because $[\n_k , \n_l] \subset \gamma_{k+l}(\n^\Q)$, the direct sum only needs to run over non-negative integers $r$. This gives 
		\begin{equation} \label{eq:subsetOfDirectSum} [\n_k, \n_l] \subset \bigoplus_{r \in \N} \n_{k + l + r\cdot d_A}.
		\end{equation}
		Now, if we assume in addition that $c \leq d_A + 1$, then we have for any $1 \leq k, l \leq c$ that $\n_{k + l + r\cdot d_A} = 0$ for $r > 0$. Therefore Equation (\ref{eq:subsetOfDirectSum}) becomes $[\n_k, \n_l] \subset \n_{k+l}$ which shows that in this case $\n^\Q = \n_1 \oplus \ldots \oplus \n_c$ is a positive grading for $\n^\Q$.
	\end{proof}
	
	
	Using tools of Galois theory, we find more information about the rank of an Anosov automorphism. After ordering the roots of an irreducible polynomial of degree $n$, the Galois group of its splitting field $E$ can be seen as a subgroup of $S_n$ as explained in Section \ref{sec:prelimGalois}. For any finite set $X$ we write $S_X$ for the permutation group on $X$. It is clear that if $X \subset \{ 1, \ldots, n \}$ there is a natural inclusion of $S_X$ in $S_n$ by extending a permutation of $X$ by the identity on the complement. The following lemma from \cite[Lemma 3.7.]{payn09-1} gives more information about the subgroup of $\gal(E, \Q)$ that fixes a certain element of $E$. We recall the proof for completeness.
	
	\begin{lemma}
		\label{lem:FullRankPermutationNotTransitive}
		Let $f$ be an irreducible Anosov polynomial which satisfies the full rank condition. Let $E$ denote the splitting field of $f$ and $\lambda_1, \ldots, \lambda_n$ its roots. Consider an element $\mu = \lambda_1^{e_1} \cdot\ldots\cdot \lambda_n^{e_n}$ with $e_i \in \N$ and a corresponding equivalence relation $\sim$ on $I := \{1, \ldots, n\}$ defined by $i \sim j \Leftrightarrow e_i = e_j$. Then we have
		\begin{align*}
			\Aut(E, \Q(\mu)) \subseteq \bigoplus_{[i] \in I/\sim} S_{[i]}.
		\end{align*}
		By consequence if $\mu \neq \pm 1$, $\Aut(E, \Q(\mu))$ does not act transitively on the roots $\lambda_1, \ldots, \lambda_n$.
	\end{lemma}
	\begin{proof}
		Take any $\sigma \in \Aut(E, \Q(\mu))$, so $\sigma$ fixes $\mu$ and we have
		\begin{align*}
			&\sigma(\lambda_1^{e_1} \cdot \ldots \cdot \lambda_n^{e_n}) = \lambda_1^{e_1} \cdot \ldots \cdot \lambda_n^{e_n}\\
			\Rightarrow \quad & \lambda_{\sigma(1)}^{e_1} \cdot \ldots \cdot \lambda_{\sigma(n)}^{e_n} = \lambda_1^{e_1} \cdot \ldots \cdot \lambda_n^{e_n}\\
			\Rightarrow \quad & \lambda_1^{e_{\sigma^{-1}(1)} - e_1} \cdot \ldots \cdot \lambda_n^{e_{\sigma^{-1}(n)} - e_n} = 1.
		\end{align*}
		Since the $\lambda_i$ have full rank, this implies that
		\begin{equation}
			\label{eq:exponentsFullRank}
			e_{\sigma^{-1}(1)} - e_1 = \ldots = e_{\sigma^{-1}(n)} - e_n.
		\end{equation}
		Then let $(i_1 \cdots i_k)$ be any of the disjoint cycles of $\sigma$ seen as an element of $S_n$. Then it follows from equation (\ref{eq:exponentsFullRank}) that $e_{i_1} = \ldots = e_{i_k}$. So $i_1, \ldots, i_k \in [i_1]$ and by consequence $(i_1 \cdots i_k) \in S_{[i_1]}$. We thus see that $\sigma$ is a composition of elements which each lie in some $S_{[i]}$ and thus this proves the claim.
		
		For the last statement, we know that if $\Aut(E,\Q(\mu))$ acts transitively, it must hold that $i \sim j$ for all $i$ and $j$ and thus that $e_i = e_j$ for all $i$ and $j$. In particular, we have that $\mu = \left( \lambda_1 \cdot \ldots \cdot \lambda_n\right)^{e_1} = (\pm1)^{e_1} = \pm1$.
	\end{proof}
	
	We can now use this lemma to prove the following proposition \cite[Corollary 3.9.]{payn09-1}, which gives us more information about the type of an Anosov Lie algebra.
		
	\begin{prop}
	    \label{prop:n1PrimeFullRankThenDividesni}
	    Let $\n^\Q$ be a nilpotent Lie algebra with $n_1$ prime and $A: \n^\Q \to \n^\Q$ an Anosov automorphism of full rank. Then $n_1$ divides $n_i$ for all $2 \leq i \leq c$.
	\end{prop}
	\begin{proof}
	    Write $f_i$ for the characteristic polynomial of $A_i$. Since $f_1$ satisfies the full rank condition, it is irreducible as well. By Lemma \ref{lem:IrrPrimeDegreeCyclicPerm} there exists an ordering of of the roots of $f$, say $\lambda_1, \lambda_2, \ldots, \lambda_{n_1}$ and an element $\sigma$ of order $n_1$ which corresponds to the permutation $(1 \, 2 \, \ldots \, n_1)$ on these roots. Each $f_i$ can be written as a product of its irreducible factors $f_i = g_{i1} \cdot g_{i2} \cdot \ldots \cdot g_{ik}$. Let $M_{ij}$ denote the set of roots of the polynomial $g_{ij}$. It is clear that $\gal(E, \Q)$ has an action on this set and that $|M_{ij}| = \deg g_{ij}$. Now take an element $\mu \in g_{ij}$. By Lemma \ref{lem:FullRankPermutationNotTransitive} we know that $\sigma$ can not be an element of $\Aut(E, \Q(\mu))$, meaning that $\sigma(\mu) \neq \mu$. The orbit of $\mu$ under the subgroup of $\gal(E, \Q)$ generated by $\sigma$, written $\langle \sigma \rangle$ must therefore count at least two elements. Since $\sigma$ has prime order, the orbit stabilizer theorem then tells us that this orbit counts exactly $n_1$ elements. This proves that all the orbits under $\langle \sigma \rangle$ in $M_{ij}$ count $n_1$ elements. Since these orbits form a partition of $M_{ij}$, this proves that $n_1 \mid M_{ij} = \deg g_{ij}$. By consequence $n_1 \mid \deg f_i = n_i$ for all $2 \leq i \leq c$.
	\end{proof}


	The final lemma shows that we can assume that the irreducible components of Anosov polynomials of rank $1$ have degree $2$ by taking a finite  power.

	\begin{lemma}
	    \label{lem:AnosovPolyRank1}
		Let $A$ be an Anosov automorphism of rank one. Up to taking a power of $A$ its characteristic polynomial is a product of Anosov polynomials of degree 2.
	\end{lemma}
	\begin{proof}
		Let $g$ be an irreducible factor of $f$. It follows that $g$ is Anosov of rank 1 as well. Let $\lambda_1, \ldots, \lambda_k$ be the roots of $g$. Since they generate an abelian group of rank 1, there exist non-zero integers $e_i$ such that
		\begin{align*}
			\lambda_1^{e_1} = \lambda_2^{e_2} = \ldots = \lambda_n^{e_n}.
		\end{align*}
		Since $g$ is irreducible, there exists for each index $2 \leq i \leq n$ a field automorphism $\sigma_i$ in the Galois group of $f$ over $\Q$ such that $\sigma_i(\lambda_1) = \lambda_i$. This implies that $\sigma_i(\lambda_1^{e_i}) = \lambda_i^{e_i} = \lambda_1^{e_1}$. Write $s_i$ for the order of $\sigma_i$, then it follows that
		\begin{equation*}
		\lambda_1^{\left(e_i^{s_i}\right)} = \sigma_i^{s_i}\left(\lambda_1^{\left(e_i^{s_i}\right)}\right) = \lambda_1^{\left(e_1^{s_i}\right)}.
		\end{equation*}
		Since $\lambda_1$ has absolute value different from $1$ we must have that $e_1^{s_i} = e_i^{s_i}$ and thus that $e_1 = \pm e_i$. Note that we can assume $e_1 = 1$ after taking the $e_1$-th power of $A$ if necessary. This can be done for every irreducible factor of $f$ since there are only finitely many of them. By consequence we get that the roots of $g$ all lie in the set $\{ \lambda_1, \lambda_1^{-1} \}$. This shows that the degree of $g$ is 2 since it is irreducible and thus must have distinct roots. We conclude that $f$ is a product of degree 2 Anosov polynomials.
	\end{proof}

\section{Positive gradings on Anosov Lie algebras}
\label{sec:dimension}
In this section, we use the notation and the results of the previous section to prove that every nilpotent Anosov Lie algebra $\n^\Q$ of dimension strictly less than 12 admits a positive grading. On the level of nilmanifolds, this implies that a nilmanifold which admits an Anosov diffeomorphism but no expanding map must have dimension at least 12. Since an Anosov diffeomorphism on an infra-nilmanifold lifts to one on the covering nilmanifold and the existence of expanding maps on infra-nilmanifolds only depends on the covering nilpotent Lie group, we find that this generalizes to infra-nilmanifolds as well. At the end of this section we will thus have proved the following.

\begin{theorem}
\label{thm:noexpanding}
    Every infra-nilmanifold of dimension $<12$ which admits an Anosov diffeomorphism also admits an expanding map.
\end{theorem}

We divide the proof into separate cases, depending on the value of $n_1$ of the type $(n_1, \ldots, n_c)$ of the Lie algebra. Since every 2-step nilpotent Lie algebra is positively graded, we only need to check Lie algebras of nilpotency class at least 3. Since by \cite[Proposition 2.3.]{lw08-1} we also know that $n_1 \geq 3$, $n_i \geq 2$ for all $2 \leq i \leq c$ and that $\dim \n^\Q < 12$, this leaves us only with the cases $n_1 = 3, 4, 5, 6, 7$. We are only interested in constructing a positive grading on the Lie algebras, therefore we do not have to determine the full Lie bracket on the Lie algebra. As before, we always assume that our Anosov automorphism $A: \n^\Q \to \n^\Q$ is semi-simple, has splitting field $E$ and satisfies $\det(A_1) = 1$.

    \subsection{Case $n_1 = 3$}
    
    Since in this case, every Anosov automorphism is of full rank, we immediately have the following result.
         
    \begin{prop}
        Let $\n^\Q$ be a nilpotent Anosov Lie algebra with $n_1 = 3$ and $\dim \n^\Q \leq 12$, then $\n$ has a positive grading. 
    \end{prop}
    \begin{proof}
        It is clear that $f_1$ is irreducible since linear factors give an eigenvalue equal to $\pm 1$. Therefore Proposition \ref{prop:primeDegFullRank} implies that $f_1$ satisfies the full rank condition. Using Proposition \ref{prop:n1PrimeFullRankThenDividesni}, we know that $3 \mid n_i$ for all $2 \leq i \leq c$ and thus that $\dim \n^\Q \geq 3\cdot c$. By our assumption on the dimension of $\n$ we thus have that $c \leq 4$. At last we use Proposition \ref{prop:CnotTooHighThenGraded} to conclude that $\n^\Q$ has a positive grading.
    \end{proof}
    
	\subsection{Case $n_1 = 4$}
	For this case, we will make a distinction according to the rank of the Anosov automorphism. This gives us more information about the eigenvalues by the following proposition.
	\begin{prop}
		\label{prop:possibleRanksFor4roots}
		Let $\n^\Q$ be a nilpotent Anosov Lie algebra of type $(4, n_2, \ldots, n_c)$ and let $\lambda_1$, $\lambda_2$, $\lambda_3$, $\lambda_4$ be the eigenvalues of $A_1$. Then, up to replacing $A$ by some power of $A$ and reordering the roots, we either have:
		\begin{enumerate}[label = (\roman*)]
			\item $A$ has full rank and $d_A = 4$;
			\item $A$ has rank $2$, the eigenvalues satisfy $\lambda_1 = \lambda_2^{-1}, \, \lambda_3 = \lambda_4^{-1}$ and $d_A = 2$; or
			\item $A$ has rank 1 and the eigenvalues satisfy $\lambda_1 = \lambda_2^{-1}, \, \lambda_3 = \lambda_4^{-1}, \lambda_1^k = \lambda_3^l$ with $l,k$ non-zero integers which are coprime. If $k + l$ is even then $d_A = 2$, otherwise $d_A = 1$.
		\end{enumerate}
	\end{prop}
	\begin{proof}
		 As usual, let $f_1$ denote the characteristic polynomial of $A_1$. If $A$ has full rank, then it is clear that $d_A = 4$ since $\ker \phi_A = \Z \cdot (1, 1, 1, 1)$.
		
		 If $A$ has rank 2, we get that $\rank(\ker \phi_A) = 2$. We first consider the case when $f_1$ is reducible. Under this assumption, $f_1$ factors into two irreducible polynomials of degree 2 since it does not have linear factors and thus $\lambda_1 = \lambda_2^{-1}$ and $\lambda_3 = \lambda_4^{-1}$. Now consider the case when $f_1$ is irreducible. Since $\rank (\ker \phi_A) = 2$, there exists an element $(z_1, z_2, z_3, z_4) \in \ker \phi_A$ linearly independent from $(1, 1, 1, 1)$. This implies that not all $z_i$ are equal. Up to some permutation of the indices and possibly adding an integer multiple of $(1, 1, 1, 1)$, we may also assume that $z_1 \geq z_2 \geq z_3 \geq z_4 = 0$. Since $f_1$ is irreducible, there exists an element $\sigma$ in its Galois group such that $\sigma(\lambda_3) = \lambda_4$. It follows that $(z_{\sigma(1)}, z_{\sigma(2)}, 0 , z_{\sigma(4)})$ is also an element of $\ker \phi_A$. Since $\rank( \ker \phi_A) = 2$, there exist integers $r, s, t$ such that
	\begin{align*}
		r(z_{\sigma(1)}, z_{\sigma(2)}, 0 , z_{\sigma(4)}) = s(1,1,1,1) + t(z_1, z_2, z_3, 0).
	\end{align*}
	Note that this implies that either $z_{\sigma(1)} \geq z_{\sigma(2)} \geq 0 \geq z_{\sigma(4)}$ or $z_{\sigma(1)} \leq z_{\sigma(2)} \leq 0 \leq z_{\sigma(4)}$, depending on the sign of the integer $r \cdot t$. But since we also had $z_i \geq 0$ for all $i$, this gives that at least two of the $z_i$ are zero. This shows that there is, again up to permutation of the indices, an element in the kernel of the form $(u, -v, 0, 0)$ with $u > 0,v < 0$ integers and therefore $\lambda_1^u = \lambda_2^{v}$. Since $f_1$ was assumed to be irreducible, there is an element $\tau$ in the Galois group of $f_1$ such that $\tau(\lambda_1) = \lambda_2$. Therefore we get that $\tau(\lambda_1^{v}) = \tau(\lambda_1)^v = \lambda_2^v = \lambda_1^u$. Let $|\tau|$ denote the order of $\tau$. We get that 
	\begin{align*}
		\lambda_1^{v^{|\tau|}} = \tau^{|\tau|}(\lambda_1^{v^{|\tau|}}) = \lambda_1^{u^{|\tau|}}.
	\end{align*}
	Since $|\lambda_1| \neq 1$ this implies that $u = - v$ and thus that $\lambda_1^u = \lambda_2^{- u}$. We can then assume that $u = 1$ after taking the $u$-th power of the automorphism $A$ if necessary. So in both the reducible as the irreducible case, we get that $\lambda_1 = \lambda_2^{-1}$ and $\lambda_3 = \lambda_4^{-1}$, up to taking a power of $A$. Combined with the assumption that $\rank(\ker \phi_A) = 2$, this implies that $\ker \phi_A$ is generated by the elements $(1, 1, 0, 0)$ and $(0, 0, 1, 1)$, showing that $d_A = 2$.
		
	If $A$ has rank 1, Lemma \ref{lem:AnosovPolyRank1} tells us that, by taking a power of $A$ if necessary, we can assume that $f_1$ is the product of two Anosov polynomials of degree 2. This shows that up to changing the indices of the roots, $\lambda_1 = \lambda_2^{-1}$ and $\lambda_3 = \lambda_4^{-1}$. Since $A$ has rank 1, there must also be integers $l'$ and $k'$ such that $\lambda_1^{k'} = \lambda_3^{l'}$. If $m$ denotes the greatest common divisor of $k'$ and $l'$ it is clear that after taking the $m$-th power of $A$, the roots satisfy $\lambda_1 = \lambda_2^{-1}, \,\lambda_3 = \lambda_4^{-1}$ and $\lambda_1^k = \lambda_3^l$ with $k = k'/m$ and $l = l'/m$ coprime. Since $\rank(\ker\phi_A) = 3$ it follows that every element $(z_1, \ldots, z_4)$ in $\ker \phi_A$ must satisfy $l z_1 - lz_2 - kz_3 + kz_4 = 0$. This implies that $l(z_1 + z_2 + z_3 + z_4) = 2l z_2 + (k + l) z_3 + (l-k) z_4$. If $k + l$ is even, then $l-k$ is even and $l$ is odd. By consequence $z_1 + z_2 + z_3 + z_4$ is even and $d_A = 2$. If $k + l$ is odd, then $d_A$ can not be even since $k + l \in d_A \cdot \Z$. We also have that $\psi(1, 1, 0, 0) = 2 \in d_A \cdot \Z$ and thus it follows that $d_A = 1$. 
	\end{proof}
	
	This result tells us more about the possible Galois groups of an Anosov polynomial $f$ of degree 4. Let $\lambda_1, \ldots, \lambda_4$ denote the roots of $f$. If $f$ is reducible, it factors as two irreducible Anosov polynomials of degree 2. Since the Galois group can only permute the roots of each factor separately, we get that it is either isomorphic to $\Z_2$ or to $\Z_2 \oplus \Z_2$. If $f$ is irreducible, but does not satisfy the full rank condition, Proposition \ref{prop:possibleRanksFor4roots} tells us the relations $\lambda_1^k = \lambda_2^{-k}$ and $\lambda_3^k = \lambda_4^{-k}$ must hold for some integer $k > 0$. If the Galois group of $f$ contains an element $\sigma$ of order 3, then without loss of generality its action on the roots of $f$ is given by $\sigma(\lambda_1) = \lambda_2$, $\sigma(\lambda_2) = \lambda_3$, $\sigma(\lambda_3) = \lambda_1$ and $\sigma(\lambda_4) = \lambda_4$. This gives that
	\[ \lambda_1^k = \sigma(\lambda_3^k) = \sigma( \lambda_4^{-k}) =  \sigma^2(\lambda_4^{-k}) = \sigma^2(\lambda_3^k) = \lambda_2^k = \lambda_1^{-k}.\]
	This contradicts the fact that $\lambda_1$ has absolute value different from 1. By consequence there are no elements of order three in the Galois group of $f$, thus implying that it is isomorphic to either the cyclic group $\Z_4$, the Klein-four group $K_4$ or the dihedral group $D_4$ of order eight. We summarized this in Table \ref{tab:GaloisGroupsAnosovPolyUpTo4} below. We included the Anosov polynomials of degree two and three as well for completeness.
	
    \begin{table}[H]
    
    \begin{center}
        \begin{tabular}{ c | c | c | c | c}
        $\deg$ & full rank & irreducible & Galois group & order \\ \hline
        2 &  yes & yes & $\Z_2$ & 2\\
        3 & yes & yes & $\Z_3, S_3$ & 3, 6\\
        4 & yes & yes & $\Z_4, K_4, D_8, A_4, S_4$ & 4, 4, 8, 12, 24\\
        4 & no & yes & $\Z_4, K_4, D_8$ & 4, 4, 8\\
        4 & no & no & $\Z_2, \Z_2 \oplus \Z_2$ & 2, 4
        \end{tabular}
    \end{center}
    \caption{Possible Galois groups of Anosov polynomials up to degree 4.}
    \label{tab:GaloisGroupsAnosovPolyUpTo4}
    \end{table}
	
	This now gives the following information about the type;
	
	\begin{prop}
	    \label{prop:AnosovPolyDeg4}
	    Let $f$ be an Anosov polynomial of degree 4 with roots $\lambda_1, \lambda_2, \lambda_3, \lambda_4$ and let $e_1, e_2, e_3, e_4 \in \Z$ such that $\mu = \lambda_1^{e_1} \lambda_2^{e_2} \lambda_3^{e_3} \lambda_4^{e_4}$ is not a rational number. Then the minimal polynomial of $\mu$ over $\Q$ has even degree.
	\end{prop}
	\begin{proof}
	    Let $G$ denote the Galois group of $E$. If the order of $G$ is a power of 2, the orbit stabilizer theorem tells us that the orbit of $\mu$ under the action of $G$ must have either 1 element or an even number of elements. The former is not possible since it implies $\mu$ is rational. Therefore we conclude that the minimal polynomial of $\mu$ has an even number of roots and thus has even degree. 
	    Hence we only need to prove the statement in the case where the order of $G$ is not a power of 2. 
	    
	    From Table \ref{tab:GaloisGroupsAnosovPolyUpTo4} it follows that this only occurs when $f_1$ satisfies the full rank condition and $G$ is isomorphic to either $A_4$ or $S_4$. Again arguing as before we find by the orbit stabilizer theorem that the orbit of $\mu$ under the action of $G$ either has $3$ or an even number of elements. In the latter case we are done so assume the orbit counts 3 elements. This implies that the minimal polynomial of $\mu$, which we denote by $g$, has degree 3. The Galois group of $E$ over the splitting field of $g$ over $\Q$, denoted as $\gal(E, \Q(g))$, is a strictly smaller normal subgroup of $G$. For $S_4$ these subgroups are up to automorphism given by $A_4, K_4, \{1\}$ and for $A_4$ they are given by $K_4, \{1\}$. Note that $\gal(E, \Q(g)) \subset \Aut(E, \Q(\mu))$ can not be a transitive subgroup of $S_4$ by Lemma \ref{lem:FullRankPermutationNotTransitive}. This leaves us only with the case $\gal(E, \Q(g)) = \{1\}$, which in turn implies $\Q(g) = E$. This is impossible since by our assumption $g$ had degree 3 and thus $[\Q(g) : \Q] \leq 6$ whereas $[E:\Q] \geq 12$.
	\end{proof}

	By using our techniques, we find a new and shorter proof of \cite[Theorem 1.3.]{payn09-1}.
	
	\begin{cor}
		\label{cor:first4resteven}
		If $\n^\Q$ is an Anosov Lie algebra of type $(4, n_2, \ldots, n_c)$, then $2 | n_i$ for all $2 \leq i \leq c$.
	\end{cor}
	\begin{proof}
		Let $f_i$ denote the characteristic polynomial of $A_i$. It is clear that $f_1$ is an Anosov polynomial of degree 4. The roots of $f_i$ are $i$-fold products of the roots of $f_1$. They are also not rational since the only rational algebraic units are $1$ and $-1$ which have absolute value equal to 1 contradicting the hyperbolicity of $A$. Using Proposition \ref{prop:AnosovPolyDeg4} we get that the irreducible factors of each $f_i$ must be of even degree and thus each $f_i$ has even degree itself. This shows that $n_i$ is even for all $2 \leq i \leq c$.
	\end{proof}
    
    This leads to the main result about Anosov Lie algebras of type $(4, \ldots, n_c)$.

	\begin{prop}
	\label{prop:Anosovn1is4ThenPosGraded}
		Let $\n^\Q = \n_1 \oplus\ldots\oplus \n_{c}$ be a nilpotent Anosov Lie algebra with $n_1 = 4$ and $\dim \n < 12$, then $\n^\Q$ has a positive grading.
	\end{prop}
	\begin{proof}
Since $2$-step nilpotent Lie algebras always have a positive grading, we will only consider Lie algebras of nilpotency class at least 3. Let $f_1$ be the characteristic polynomial of $A_1$ and $E$ its splitting field over $\Q$. Let $X_1, \ldots, X_4 \in \n_1^E$ be a basis of eigenvectors of $A_1$ with eigenvalues $\lambda_1, \ldots, \lambda_4$, respectively. If these roots have full rank, then $\n^\Q$ is positively graded by Proposition \ref{prop:CnotTooHighThenGraded} where we keep in mind that $\dim \n^\Q < 12$, so we focus on the case when $\lambda_1, \ldots, \lambda_4$ do not have full rank. We can assume by Proposition \ref{prop:possibleRanksFor4roots} that $\lambda_1 = \lambda_2^{-1}$ and $\lambda_3 = \lambda_4^{-1}$. Let us write for simplicity $\lambda_1 = \lambda$ and $\lambda_3 = \mu$, then the eigenvalues of $A_1$ take the form $\lambda, \lambda^{-1}, \mu, \mu^{-1}$. From here on we prove the statement separately for each possible type. Recall that by Corollary \ref{cor:first4resteven} all $n_i$ are even, leading to four possibilities for the type.
		
		\begin{itemize}
			\item \underline{Type $(4, 4, 2)$:} We have that $\n_2^E$ is spanned by the vectors $\{[X_i, X_j] \mid 1 \leq i < j \leq 4 \}$. The brackets $[X_1, X_2]$ and $[X_3, X_4]$ have to be zero, since if this is not the case they are eigenvectors of $A$ with eigenvalues $\lambda \lambda^{-1} = 1$ and $\mu  \mu^{-1} = 1$, respectively. This would contradict the hyperbolicity of $A$. Because $\dim \n_2 = 4$ it follows that a basis for $\n_2^E$ can be given by the vectors $[X_1, X_3]$, $[X_2, X_4]$, $[X_1,X_4]$ and $[X_2, X_3]$. Therefore we have that $[\n_1, \n_1] \subset \n_2$, which implies that $\n = \n_1 \oplus \n_2 \oplus \n_3$ is a positive grading for $\n^\Q$.
			
			\item \underline{Type $(4,2,4)$:} The eigenvalues of $A_2$ must be 2-fold products of the eigenvalues of $A_1$. Without loss of generality we can assume that one eigenvalue of $A_1$ is given by $\lambda \mu$. It follows that the other eigenvalue of $A_1$ is the inverse namely $\lambda^{-1} \mu^{-1}$. The brackets $[X_2, X_3]$ and $[X_1, X_4]$ must lie in $\gamma_3(\n^E)$, since otherwise their eigenvalues given by $\lambda^{-1} \mu$ and $\lambda \mu^{-1}$ respectively, must equal one of the two eigenvalues of $A_2$ which gives either that $\mu^2 = 1$ or that $\lambda^2 = 1$, a contradiction. We also know that $[X_1, X_2] = 0$ and $[X_3, X_4] = 0$ since otherwise these would be eigenvectors with eigenvalue 1. Therefore it is clear that a basis of eigenvectors of $A_2$ is given by $Y_1 := [X_1, X_3]$ and $Y_2 := [X_2, X_4]$.
			
			The eigenvectors of $A_3$ can be written as a linear combination of the elements $[X_i, Y_j]$.  Using the Jacobi identity we get that
			\begin{align*}
				[X_2, Y_1] &= [X_1, \underbrace{[X_2, X_3]}_{\in \gamma_3(\n^E)}] + [X_3, \bcancel{[X_1, X_2]}] = 0,\\
				[X_4, Y_1] &= [X_1, \bcancel{[X_4, X_3]}] + [X_3, \underbrace{[X_4, X_2]}_{\in \gamma_3(\n^E)}] = 0,\\
				[X_1, Y_2] &= [X_2, \underbrace{[X_1, X_4]}_{\in \gamma_3(\n^E)}] + [X_4, \bcancel{[X_2, X_1]}] = 0,\\
				[X_3, Y_2] &= [X_2, \bcancel{[X_3, X_4]}] + [X_4, \underbrace{[X_2, X_3]}_{\in \gamma_3(\n^E)}] = 0.
			\end{align*}
			This gives that $[X_1, Y_1]$, $[X_3, Y_1]$, $[X_2, Y_2]$ and $[X_4, Y_2]$ form a basis of eigenvectors for $A_3$. Its eigenvalues are therefore given by $\lambda^2 \mu$, $\lambda \mu^2$, $\lambda^{-2} \mu^{-1}$ and $\lambda^{-1}\mu^{-2}$, respectively. 
			
			We know that $[X_2, X_3], [X_1, X_4] \in \n_3^E$. If they are non-zero then they must be eigenvectors of $A_3$ and thus this gives one of the following equalities:
			\begin{align*}
				\lambda^{-1} \mu = \begin{cases}
					\lambda^2 \mu &\Rightarrow \lambda^3 = 1\\
					\lambda \mu^2 &\Rightarrow \lambda^2 \mu = 1\\
					\lambda^{-2} \mu^{-1} &\Rightarrow \lambda \mu^2 = 1\\
					\lambda^{-1} \mu^{-2} &\Rightarrow \mu^3 = 1.
				\end{cases}, \quad \quad \quad 
				\lambda \mu^{-1} = \begin{cases}
					\lambda^2 \mu &\Rightarrow \lambda \mu^2 = 1\\
					\lambda \mu^2 &\Rightarrow \mu^3 = 1\\
					\lambda^{-2} \mu^{-1} &\Rightarrow \lambda^3 = 1\\
					\lambda^{-1} \mu^{-2} &\Rightarrow \lambda^2 \mu = 1.
				\end{cases}
			\end{align*}
			which gives in each case a contradiction, either because $\lambda$ or $\mu$ would be a root of unity or because an eigenvalue of $A_3$ would be equal to 1. So we must conclude that $[X_2, X_3] = [X_1, X_4] = 0$. By consequence $[\n_1,  \n_1] \subset \n_2$ and thus the decomposition $\n = \n_1 \oplus \n_2 \oplus \n_3$ gives a positive grading.
			\item \underline{Type $(4,2,2)$:} Just as in the previous case, we find that $Y_1 := [X_1, X_3]$ and $Y_2 := [X_2, X_4]$ form a basis of eigenvectors for $A_2$, with eigenvalues $\lambda \mu$ and $\lambda^{-1} \mu^{-1}$, respectively. Moreover, an identical argument as before shows that $[X_2, Y_1] = [X_4, Y_1] = [X_1, Y_2] = [X_3, Y_2] = 0$, implying that the eigenvalues on $\n_3^E$ are, by interchanging $\lambda$ and $\mu$ if necessary, of the form $\lambda^2 \mu$ and $\lambda^{-2} \mu^{-1}$, with corresponding eigenvectors $Z_1 := [X_1,Y_1]$ and $Z_2 := [X_2,Y_2]$. 
			
			If $[X_1, X_4]$ is non-zero, it is an eigenvector for eigenvalue $\lambda \mu^{-1}$, which is different from $\lambda^{-2} \mu^{-1}$ because otherwise $\lambda$ would have absolute value $1$. In particular, $[X_1, X_4] = a [X_1, Y_1]$ for some $a \in E$. By replacing the vector $X_4$ by $X_4 - a Y_1$, which is again an eigenvector for the same eigenvalue, we get that $[X_1,X_4] = 0$, but with all the other relations identical. Similarly, we can replace $X_3$ to achieve $[X_2, X_3] = 0$, showing that $[\n_1^E,\n_1^E] = \n_2^E$ and thus realizing a positive grading $\n_1^E \oplus \n_2^E \oplus \n_3^E$ for $\n^E$, which in turn implies $\n^\Q$ admits a positive grading. In fact, this argument shows that $\n^E$ is a direct sum of two filiform Lie algebras, one generated by the elements $X_1, X_3$ and the other by $X_2,X_4$. Note that if $[X_3,Y_1]$ is non-zero, it is an eigenvector of eigenvalue $\lambda \mu^2$, which must be different from $\lambda^{-2} \mu^{-1}$ and thus $[X_3,Y_1] = b [X_1,Y_1]$ for $b \in E$. By replacing $X_3$ by $X_3 - b X_1$, we can thus assume that $[X_3,Y_1] = 0$. Similarly, we realize the assumption $[X_4,Y_2] = 0$.
			
			\item \underline{Type $(4,2,2,2)$:} We use the same notations and basis for the spaces $\n_1^E \oplus \n_2^E \oplus \n_3^E$ as for type $(4,2,2)$, in particular with the property that $[X_1,X_4] \in \n_4^E$ and $[X_2,X_3] \in \n_4^E$. By using the Jacobi identity, we see that $[X_2, Z_1] = [X_3, Z_1] = [X_4, Z_1] = [X_1, Z_2] = [X_3, Z_2] = [X_4, Z_2] = 0$. In particular, a basis of eigenvectors for $A_4$ is given by $[X_1,Z_1]$ and $[X_2,Z_2]$ with eigenvalues $\lambda^3 \mu$ and $\lambda^{-3} \mu^{-1}$. Now, if $[X_1,X_4]$ is non-zero, it lies in $\n_4$ and has eigenvalue $\lambda \mu^{-1}$ which is different from $\lambda^{-3} \mu^{-1}$. In particular, $[X_1,X_4] = a [X_1,Z_1]$ for some $a \in E$, and thus by replacing $X_4$ by $X_4 - a Z_1$ we get that $[X_1,X_4] = 0$. In a similar fashion, we can assume $[X_2,X_3] = 0$ and thus $[\n_1,\n_1] = \n_2$, showing that the decomposition $\n_1^E \oplus \n_2^E \oplus \n_3^E \oplus \n_4^E$ gives a positive grading for $\n^E$. In fact, we have shown that $\n^E$ is a direct sum of two filiform Lie algebras. 		\end{itemize}	\end{proof}
	Due to the proof for Anosov Lie algebras of type $(4,2,2)$ and $(4,2,2,2)$, one might conjecture that for every Anosov Lie algebra $\n^\Q$ of type $(4,2,\ldots, 2)$ it holds that $\n^E$ is isomorphic to the direct sum of two filiform Lie algebras. The family of examples in Section \ref{sec:familyofexamples} shows that this is not the case, since they have no positive grading. 
	
	\subsection{Case $n_1= 5$}
	
We first consider the possibilities if the polynomial $f_1$ is reducible.
	\begin{prop}
		\label{prop:n1is5fisreducible}
		Let $\n^\Q = \n_1 \oplus \ldots \oplus \n_c$ be an Anosov Lie algebra with $n_1 = 5$ and $c \leq 4$. If the characteristic polynomial of $A_1$ is reducible and $\n^\Q$ has no non-trivial abelian factor, then $n_2 \geq 6$.
	\end{prop}

	\begin{proof}
	Denote the characteristic polynomial of $A_1$ by $f_1$. Because $f_1$ is reducible it is equal to a product $f_1 = g_1 g_2$ with $g_1$ and $g_2$ irreducible of degree $3$ and $2$, respectively. After squaring the Anosov automorphism $A$ if necessary we can assume that the constant terms of $g_1$ and $g_2$ are equal to $1$. Let $E, E_1$ and $E_2$ denote the splitting fields of $f, g_1$ and $g_2$, respectively. Because $E_1 \subset E$ and $E_2 \subset E$ are subfields, we get that both $3$ and $2$ divide the order of $\gal(E, \Q)$. Let $\sigma \in \gal(E, \Q)$ be an element of order $3$. Let $\lambda_1, \lambda_2, \lambda_3$ denote the roots of $g_1$ with corresponding eigenvectors $X_1, X_2, X_3 \in \n_1^E$ and $\mu, \mu^{-1}$ the roots of $g_2$ with corresponding eigenvectors $Y_1, Y_2 \in \n_1^E$. We must have, up to reordering the $\lambda_i$, that
		\begin{align*}
			\sigma(\lambda_1) = \lambda_2, \quad \sigma(\lambda_2) = \lambda_3, \quad \sigma(\lambda_3) = \lambda_1\\
			\sigma(\mu) = \mu, \quad \sigma(\mu^{-1}) = \mu^{-1}.
		\end{align*}
		This follows from the fact that $\sigma$ must permute the roots of each irreducible polynomial separately and from the orbit-stabilizer theorem which implies that the orbit of an element under $\sigma$ must have either $1$ or $3$ elements. 
		
		We first show that $[X_1,Y_1] \in  \n_2^E$. Indeed, if the vector $[X_1, Y_1]$ is non-zero, it is an eigenvector of $A$ with eigenvalue $\lambda_1 \mu$. All the eigenvalues of $A_k$ are $k$-fold products of eigenvalues of $A_1$. So if the eigenvalue $\lambda_1 \mu$ occurs on $\n_k^E$, we must have $\lambda_1 \mu = \lambda_1^{e_1} \lambda_2^{e_2} \lambda_3^{e_3} \mu^s$ for some positive integers $e_i, s \in \N$ with $e_1 + e_2 + e_3 + s = k$. This implies that $\mu^{1 - s} = \lambda_1^{e_1-1} \lambda_2^{e_2} \lambda_3^{e_3}$. Clearly the left hand side is invariant under $\sigma$ so the same holds for the right hand side. By Proposition \ref{prop:primeDegFullRank} the roots $\lambda_1, \lambda_2, \lambda_3$ satisfy the full rank condition. Using that $\lambda_1^{e_1-1} \lambda_2^{e_2} \lambda_3^{e_3}$ is invariant under $\sigma$, Lemma \ref{lem:FullRankPermutationNotTransitive} implies that $e_1 - 1 = e_2 = e_3$. From this we get that $\mu^{1 - s} = 1$ and thus since $\mu$ is not a root of unity that $s = 1$. By consequence we also have that $k = e_1 + e_2 + e_3 + s = e_1 + (e_1 - 1) + (e_1 - 1) + 1 = 3e_1 - 1$. Since $c$ is assumed to be less or equal than 4 this proves that $k = 2$ and thus $[X_1, Y_1] \in \n_2^E$. Analogously it can be proven that $[X_i, Y_j] \in \n_2^E$ for all $1 \leq i \leq 3$ and $1 \leq j \leq 2$. 
		
		Therefore we know that either all the brackets $[X_i, Y_j]$ are zero or that $A_2$ has an eigenvalue of the form $ \lambda_i \mu^{\pm 1}$. In the first case $\text{span}_E\{Y_1, Y_2\}$ is an abelian factor of $\n^E$. By consequence the rational Lie algebra $\n^\Q$ has a non-trivial abelian factor as well which is in contradiction with the assumption. We thus have without loss of generality, that $\lambda_1 \mu$ is an eigenvalue of $A_2$. By letting $\sigma$ act on this we get that $\lambda_1 \mu, \lambda_2 \mu, \lambda_3 \mu$ are distinct eigenvalues of $A_2$. Their product is equal to $\lambda_1 \lambda_2 \lambda_3 \mu^3 = \mu^3$ which can not be equal to 1 since $\mu$ is not a root of unity. Therefore $A_2$ must have at least one other eigenvalue. This can be either one of the form $\lambda_i \lambda_j$ or of the form $\lambda_i \mu^{-1} $. In either case, by letting $\sigma$ act on these roots we see they each have at least 2 other conjugates. By consequence $n_2$ must be greater or equal than 6.
	\end{proof}
	
	As a consequence, we can study positive gradings on Anosov Lie algebras of type $(5, \ldots)$.
	
	\begin{cor}
		Let $\n^\Q = \n_1 \oplus\ldots\oplus \n_{c}$ be a nilpotent Anosov Lie algebra with $n_1 = 5$ and $\dim \n^\Q < 12$, then $\n^\Q$ has a positive grading.
	\end{cor}

	\begin{proof}
		Let $f_1$ be the characteristic polynomial of $A_1$. By Proposition \ref{prop:CnotTooHighThenGraded} we know that if the roots of $f_1$ have full rank, then the Lie algebra $\n^\Q$ has a positive grading. In the other case, if the roots of $f_1$ are not of full rank, then Proposition \ref{prop:primeDegFullRank} implies that $f_1$ is reducible. We can also assume that $\n^\Q$ has no non-trivial abelian factor since otherwise our work in the previous sections and Theorem \ref{thm:directSumAnosov} and \ref{thm:directSumExp} assure that $\n^\Q$ has a positive grading. By consequence we can apply Proposition \ref{prop:n1is5fisreducible} and thus $n_2 \geq 6$. This implies $c = 2$ because of the assumption on our dimension. In particular, $\n$ has a positive grading. 
	\end{proof}

	\subsection{Case $n_1 = 6$}
	
	Since for any $2 \leq i \leq c$ we have that $n_i \geq 2$, the only possible types of dimension $< 12$ we need to check are $(6, 2, 2)$, $(6, 2, 3)$ and $(6, 3, 2)$.
	
	For types $(6, 2, 2)$ and $(6, 2, 3)$, we will use the fact that for these types the quotient by the third ideal in the upper central series gives an Anosov Lie algebra of type $(6, 2)$ which have been classified in \cite{lw09-1}. Let $\h_3$ denote the Heisenberg Lie algebra of dimension 3. If $\n^\Q$ is an Anosov Lie algebra of type (6, 2), there are, up to isomorphism, only two possibilities for the Lie algebra $\n^E$. The first possibility is $\mathfrak{h}_3^E \oplus \mathfrak{h}^E_3 \oplus E^2$ with $\h_3^E$ the $3$-dimensional Heisenberg Lie algebra over the field $E$, whereas the second possibility, denoted as $\mathfrak{g}$ in \cite{lw09-1}, is spanned by $X_1,\ldots, X_6, Z_1, Z_2$ with the Lie bracket defined by
	\begin{equation}
	\label{eq:defLieAlgebrag}
	    [X_1, X_2] = Z_1, \quad [X_1, X_3] = Z_2, \quad [X_4, X_5] = Z_1, \quad [X_4, X_6] = Z_2.
	\end{equation}
	We immediately state a result about possible extensions of the Lie algebra $\mathfrak{g}$.
	\begin{lemma}
	\label{lem:gCanNotBeIsoToQuotient}
	    Let $\n^E$ be a Lie algebra of nilpotency class at least 3 defined over some field $E$. Then $\g^E$ can not be isomorphic to $\n^E / \gamma_3(\n^E)$.
	\end{lemma}
	\begin{proof}
	    Suppose that there is an isomorphism $\mathfrak{g}^E \approx \n^E/ \gamma_3(\n^E)$. Using this isomorphism we can identify $\g^E$ with a complementary subspace of $\gamma_3(\n^E)$ in $\n^E$. In this way we get $\n^E = \text{span}_E \{X_1,\ldots, X_6, Z_1, Z_2\} \oplus \gamma_3(\n^E)$. The relations from (\ref{eq:defLieAlgebrag}) then still hold modulo $\gamma_3(\n^E)$. 
	    
	    Now take any three-fold bracket $[X_i, [X_j, X_k]]$ with $1 \leq i,j,k \leq 6$, then we claim it must lie in $\gamma_4(\n^E)$. If $j \in \{1, 2, 3\}$ and $k \in \{4, 5, 6\}$ or the other way around, we have that $[X_j, X_k] \in \gamma_3(\n^E)$. By consequence $[X_i, [X_j, X_k]] \in \gamma_4(\n^E)$. So we can assume without loss of generality that both  $j,k \in \{1, 2, 3\}$. If $i \in \{1, 2, 3\}$ as well, there exists a $Y \in \gamma_3(\n^E)$ such that $[X_j, X_k] =  [X_{j + 3}, X_{k + 3}] + Y$. Thus we have $[X_i, [X_j, X_k]] = [X_i, Y] + [X_i, [X_{j+3}, X_{k + 3}]]$. Since $[X_i, Y]$ lies in $\gamma_4(\n^E)$ it now suffices to show that $[X_i, [X_{j+3}, X_{k + 3}]] \in \gamma_4(\n^E)$. So it suffices to show that $[X_i,[X_j,X_k]] \in \gamma_4(\n^E)$ for $i \in \{4, 5, 6\}$ and $j,k \in \{1, 2, 3\}$. From the Jacobi identity we then get
	    \begin{align*}
	        [X_i, [X_j, X_k]] = [X_j, \underbrace{[X_i, X_k]}_{\in \gamma_3(\n^E)}] + [X_k, \underbrace{[X_j, X_i]}_{\in \gamma_3(\n^E)}] \in \gamma_4(\n^E).
	    \end{align*}
	    This proves that any bracket of the form $[X_i,[X_j, X_k]]$ with $1 \leq i,j,k \leq 6$ lies in $\gamma_4(\n^E)$ and thus that $\gamma_3(\n^E) = \{0\}$. Since we assumed $\n^E$ to be at least 3-step nilpotent, this gives a contradiction.
	\end{proof}
	
	We apply this result to show the existence of positive gradings.

	\begin{prop}
	    There are no Anosov Lie algebras of type $(6, 2, 3)$ and an Anosov Lie algebra of type $(6, 2, 2)$ admits a positive grading.
	\end{prop}
	\begin{proof}
	    Let $\n^\Q$ be an Anosov automorphism of type $(6, 2, n_3)$ and let $A:\n^\Q \to \n^\Q$ be the semi-simple Anosov automorphism. As usual we get the decomposition $\n^\Q = \n_1 \oplus \n_2 \oplus \n_3$ such that $A( \n_i )= \n_i$ and we write $A_i = A|_{\n_i}$. The Anosov automorphism $A$ descends to one on $\n^\Q/\gamma_3(\n^\Q)$ which we call $\tilde{A}$. The Lie algebra $\n^\Q/\gamma_3(\n^\Q)$ is thus also an Anosov Lie algebra of type $(6, 2)$. We use the natural identification of vector spaces $\n^\Q/\gamma_3(\n^\Q) \approx \n_1^\Q \oplus \n_2^\Q$ and see that under this identification $\tilde{A}$ on $\n^\Q/\gamma_3(\n^\Q)$ corresponds to $A_1 \oplus A_2$ on $\n_1 \oplus \n_2$. Denote by $E$ the splitting field of the characteristic polynomial of $A$. In \cite{lw09-1} it is proven that the Lie algebra $\n^E/\gamma_3(\n^E)$ is isomorphic to either $\g^E$ or $\h_3^E \oplus \h_3^E \oplus E^2$. From Lemma \ref{lem:gCanNotBeIsoToQuotient} we get that $\n^E/\gamma_3(\n^E)$ must be isomorphic to $\h_3^E \oplus \h_3^E \oplus E^2$, moreover by following the proof, we get a basis of eigenvectors $X_1, \ldots,X_4, Y_1, Y_2$ of $A_1$ and a basis of eigenvectors $Z_1, Z_2$ of $A_2$ such that the following relations hold in $\n^E$:
	    \begin{align}
	        &\label{eq:bracketRelations1} [X_1, X_2] = Z_1 \mod \gamma_3(\n^E)\\        &\label{eq:bracketRelations2} [X_3, X_4] = Z_2 \mod \gamma_3(\n^E)\\
	        &[X_1, X_3], [X_1, X_4], [X_2, X_3], [X_2, X_4] \in \gamma_3(\n^E)\\
	        &[\R Y_i, \n_1 \oplus \n_2] \subset \gamma_3(\n^E) \quad \quad \forall \quad i \in \{1,2\}.
	    \end{align}
	    Let $\lambda_1, \ldots, \lambda_4, \mu_1, \mu_2$ be the eigenvalues of $X_1, \ldots, X_4, Y_1, Y_2$, respectively. We know by the proof in \cite{lw09-1} as well that the characteristic polynomial $f_1$ of $A_1$ factors as $f_1 = gh$ with $\lambda_1, \ldots, \lambda_4$ the roots of $g$ and $\mu_1, \mu_2$ the roots of $h$. Now, we use the fact that $[\n^E, \gamma_3(\n^E)] = 0$ and the Jacobi identity to find
	    \begin{alignat*}{2}
	        [Y_i, Z_j] &= [Y_i, [X_{2j - 1}, X_{2j}]] = [X_{2j - 1}, \underbrace{[Y_i, X_{2j}]}_{\in \gamma_3(n^E)}] + [X_{2j}, \underbrace{[X_{2j - 1}, Y_i]}_{\in \gamma_3(n^E)}] = 0 \quad \quad &&\forall \quad i,j \in \{1, 2\}\\
	        [X_i, Z_2] &= [X_i, [X_3, X_4]] = [X_3, \underbrace{[X_i, X_4]}_{\in \gamma_3(n^E)}] + [X_4, \underbrace{[X_3, X_i]}_{\in \gamma_3(n^E)}] = 0 \quad \quad &&\forall \quad i \in \{1, 2\}\\
	        [X_i, Z_1] &= [X_i, [X_1, X_2]] = [X_1, \underbrace{[X_i, X_2]}_{\in \gamma_3(n^E)}] + [X_2, \underbrace{[X_1, X_i]}_{\in \gamma_3(n^E)}] = 0 \quad \quad &&\forall \quad i \in \{3, 4\}.
	    \end{alignat*}
	    Since the brackets above vanish, it follows that $\n_3$ is spanned by the vectors $[X_1, Z_1]$, $[X_2, Z_1]$, $[X_3, Z_2]$ and $[X_4, Z_2]$. This implies that the eigenvalues of $A_3$ lie in the set $\{ \lambda_1^2 \lambda_2, \lambda_1\lambda_2^2, \lambda_3^2 \lambda_4, \lambda_3 \lambda_4^2 \}$. Applying Proposition \ref{prop:AnosovPolyDeg4} on the Anosov polynomial $g$ and the eigenvalues of $A_3$ we see that $n_3$ must be even. By consequence a Lie algebra of type $(6, 2, 3)$ can not be Anosov. We can thus from here assume that $n_3 = 2$. 
	    
	    From (\ref{eq:bracketRelations1}) and (\ref{eq:bracketRelations2}) it follows that $Z_1$ has eigenvalue $\lambda_1 \lambda_2$ and $Z_2$ has eigenvalue $\lambda_3 \lambda_4$. If $\lambda_1^2 \lambda_2$, $\lambda_1 \lambda_2^2$ or $\lambda_3^2 \lambda_4$, $\lambda_3 \lambda_4^2$ are the eigenvalues of $A_3$ we get that $(\lambda_1 \lambda_2)^3 = 1$ or $(\lambda_3 \lambda_4)^3 = 1$ respectively. This contradicts the fact that the eigenvalues of $A_2$ can not have absolute value equal to 1. By consequence we have without loss of generality that the eigenvalues of $A_3$ are $\lambda_1^2 \lambda_2$ and $\lambda_3^2 \lambda_4$. This shows that $\lambda_1 \lambda_3 = 1$ and thus also that $\lambda_2 \lambda_4 = 1$. From here on we can argue in similarly as in the proof of type $(4, 2, 2)$ in Proposition \ref{prop:Anosovn1is4ThenPosGraded} and get that the subalgebra $\text{span}_E\{X_1, \ldots, X_4\} \oplus \n_2^E \oplus \n_3^E \subset \n^E$ has a positive grading given by $V_1 \oplus \n_2^E \oplus \n_3^E$ for some vector subspace $V_1 \subset \text{span}_E\{X_1, \ldots, X_4\} \oplus \n_2^E$. It is then straightforward to check that $W_1 = V_1$, $W_2 = \n_2^E \oplus \text{span}_E\{ Y_1, Y_2 \}$, $W_3 = \n_3^E$ defines a positive grading for $\n^E$.
	\end{proof}
	
	This leaves only one possible type, which we check via the action of the Galois group on the eigenvalues.

	\begin{prop}
	    Let $\n^\Q$ be an Anosov Lie algebra of type $(6, 3, 2)$, then $\n^\Q$ admits a positive grading.
	\end{prop}
	\begin{proof}
	    Let $\lambda_1, \ldots, \lambda_6$ denote the eigenvalues of $A_1$ with corresponding eigenvectors $X_1, \ldots, X_6$ and $\mu_1, \mu_2, \mu_3$ the eigenvalues of $A_2$ with corresponding eigenvectors $Y_1, Y_2, Y_3$. As usual, we write $E$ for the splitting field of the characteristic polynomial $f_1$ of $A_1$. Note that $f_2$, the characteristic polynomial of $A_2$, is irreducible of degree 3 since otherwise it would have a rational root. By applying the orbit stabilizer theorem to the action of $\gal(E, \Q)$ on $\mu_1$, it follows that $3$ divides the order of $\gal(E, \Q)$. By consequence there is an element $\sigma \in \gal(E, \Q)$ of order 3. The field automorphism $\sigma$ is completely determined by its action on the roots of $f_1$. Without loss of generality, there are two possibilities for $\sigma$, where we use the notation introduced in Section \ref{sec:prelimGalois}.

	    \begin{itemize}
	        \item \underline{$\sigma = (1 \, 2\, 3)(4\, 5\, 6)$} We first show that $A_3$ has no eigenvalue of the form $\lambda_i \lambda_j$ with $i \neq j$ and either $i,j \in \{1, 2, 3\}$ or $i,j \in \{4, 5, 6 \}$. Without loss of generality we can take $i = 1$, $j = 2$. We must have that $\lambda_1 \lambda_2$ is fixed under $\sigma$ since otherwise $A_3$ has 3 distinct eigenvalues which is in contradiction with $n_3 = 2$. Thus it follows that $\lambda_1 \lambda_2 = \lambda_2 \lambda_3 = \lambda_3 \lambda_1$ or with other words $\lambda_1 = \lambda_2 = \lambda_3$. This implies that $f$ is a product of 3 irreducible polynomials of degree 2 which is in contradiction with the way $\sigma$ acts on $\{\lambda_i\}_i$. This shows $A_3$ can not have an eigenvalue of this form. 
	        
	        Next we show that $A_3$ has no eigenvalue of the form $\lambda_i \lambda_j$ with $i \in \{ 1, 2, 3 \}$ and $j \in \{ 4, 5, 6 \}$. Without loss of generality we can assume $i = 1$ and $j = 4$. Again we must have that $\lambda_1 \lambda_4$ is fixed under $\sigma$ since $n_3 = 2$. Thus we get $\lambda_1 \lambda_4 = \lambda_2 \lambda_5 = \lambda_3 \lambda_6$. Therefore $1 = \lambda_1 \lambda_2 \lambda_3 \lambda_4 \lambda_5 \lambda_6 = (\lambda_1 \lambda_4)^3$ and thus $\lambda_1 \lambda_4$ has norm 1. It can therefore not be an eigenvalue of $A_3$. 
	        
	        We have shown above that $A_3$ can not have any eigenvalue of the form $\lambda_i \lambda_j$ with $1 \leq i < j \leq 6$. By consequence we have $[\n_1^E, \n_1^E] \subset \n_2^E$ and thus $\n_1 \oplus \n_2 \oplus \n_3$ is a positive grading for $\n^\Q$.
	    
	        \item \underline{$\sigma = (1\, 2\, 3)(4)(5)(6)$}. By a similar argument as in the previous case, it holds that $\lambda_i \lambda_j$ with $i \neq j$, $i \in \{1,2,3\}$ and $j \in \{1, 2, 3, 4, 5, 6\}$, is not an eigenvalue of $A_3$. We thus know that $A_2$ must have an eigenvalue of this form, otherwise $\{X_1, X_2, X_3\}$ spans an abelian factor of $\n^E$ and thus $\n^\Q$ would be a direct sum of the 3-dimensional abelian Lie algebra and an Anosov Lie algebra of type $(3, 3, 2)$ which does not exist. So $A_2$ has an eigenvalue of the form $\lambda_i \lambda_j$ with eigenvector $[X_i, X_j]$ where $i \in \{1, 2, 3\}$. Without loss of generality we can assume this eigenvalue is $\lambda_1 \lambda_2$ or $\lambda_1 \lambda_4$. We treat each case separately. 
	        
	        In the first case, as we argued before, $\sigma$ can not fix $\lambda_1 \lambda_2$ and thus the other eigenvalues of $A_2$ are given by $\lambda_2 \lambda_3$ and $\lambda_3 \lambda_1$. This shows as well that $(\lambda_1 \lambda_2 \lambda_3)^2 = 1$. Without loss of generality we can assume $A_3$ has an eigenvalue of the form $\lambda_1 \lambda_2 \lambda_i$. If $i \in \{4, 5, 6\}$, then $\sigma$ does not fix $\lambda_1 \lambda_2 \lambda_i$ which is in contradiction with $n_3 = 2$. So we must have $i \in \{1, 2, 3\}$. If $i = 1$ or $i = 2$ we get, since $\sigma$ must fix $\lambda_1 \lambda_2 \lambda_i$, that $\lambda_1^2 \lambda_2 = \lambda_2^2 \lambda_3 = \lambda_3^2 \lambda_1$ or that $\lambda_1 \lambda_2^2 = \lambda_2 \lambda_3^2 = \lambda_3 \lambda_1^2$, respectively. In either case we can derive that $\lambda_1^2 = \lambda_1 \lambda_2$, $\lambda_2^2 = \lambda_1 \lambda_3$ and $\lambda_3^2 = \lambda_1 \lambda_2$. Therefore we have
	        \begin{equation*}
	            \lambda_1^8 = \lambda_2^4 \lambda_3^4 = (\lambda_1 \lambda_3)^2 (\lambda_1 \lambda_2)^2 = (\lambda_1 \lambda_2 \lambda_3)^2 \lambda_1^2 = \lambda_1^2.
	        \end{equation*}
	        By consequence $\lambda_1^6 = 1$, which is in contradiction with the fact that $\lambda_6$ has absolute value different from 1. We thus conclude that only $\lambda_1 \lambda_2 \lambda_3$ can be an eigenvalue of $A_3$ which contradicts the fact that $n_3 = 2$. 
	        
	        Now consider the second case where $\lambda_1 \lambda_4$ is an eigenvalue of $A_2$. The other eigenvalues of $A_2$ are then $\lambda_2 \lambda_4$ and $\lambda_3 \lambda_4$. Without loss of generality $A_3$ must have an eigenvalue of the form $\lambda_1 \lambda_4 \lambda_i$. If $i \in \{ 4, 5, 6 \}$ it follows that its orbit under $\sigma$ has three distinct elements: $\lambda_1 \lambda_4 \lambda_i, \lambda_2 \lambda_4 \lambda_i$ and $\lambda_3 \lambda_4 \lambda_i$, which can not happen since $n_3 = 2$. So we must have that $i \in \{1, 2, 3\}$, but then again its orbit under sigma counts three distinct elements contradicting $n_3 = 2$. 
	        
	        Since both cases lead to contradictions, we conclude that the case $\sigma =  (1\, 2\, 3)(4)(5)(6)$ does not occur at all, which finishes the proof.
	    \end{itemize}

	\end{proof}

	\subsection{Case $n_1 = 7$}
	As a first lemma, we show how reducibility of the polynomial $f_1$ sometimes gives us information about the type.
	\begin{lemma}
	    \label{lem:productOfPolyWithOnePrime}
	    Let $\n^\Q$ be a rational nilpotent Lie algebra with Anosov automorphism $A:\n^\Q \to \n^\Q$ and corresponding decomposition $\n^\Q = \n_1 \oplus \ldots \oplus \n_c$. Let $f_1$ be the characteristic polynomial of $A_1$ and assume its factorisation in irreducible polynomials is given by $f_1 = g \cdot h_1, \ldots, h_k$ where $\deg g = p$ is prime. We write $E_i$ for the splitting field of the polynomial $h_i$ over $\Q$. If $\n^\Q$ has no non-trivial abelian factors and $p$ does not divide the order of $\gal(E_i, \Q)$ for all $1 \leq i \leq k$, then there exists an index $2 \leq j \leq c$ such that $n_j \geq p$.
	\end{lemma}
	\begin{proof}
	    Let $\lambda_1, \ldots, \lambda_p$ be the roots of $g$ with corresponding eigenvectors $X_1, \ldots, X_p \in \n_1^E$ and $\mu_1, \ldots, \mu_{n_1 - p}$ the roots of $h_1\cdot \ldots \cdot h_k$ with eigenvectors $Y_1, \ldots, Y_{n_1-p}$. We denote by $E$ the splitting field of $f_1$. Since the Galois group $\gal(E, \Q)$ acts transitively on the roots of $g$, it follows by the orbit stabilizer theorem that $p$ divides the order of $\gal(E, \Q)$. By consequence there exists an element $\sigma \in \gal(E, \Q)$ of order $p$. By restricting $\sigma$ to the subfield $E_i$, we get an automorphism $\sigma|_{E_i} \in \gal(E_i, \Q)$ which must have order 1 or $p$. By our assumption $p$ does not divide the order of $\gal(E_i, \Q)$ and thus we must have that $\sigma|_{E_i} = \text{Id}$. So $\sigma$ acts trivially on the roots $\mu_i$. By following a similar argument as in the proof of Lemma \ref{lem:IrrPrimeDegreeCyclicPerm}, we find that $\sigma$ acts, without loss of generality, by the cyclic permutation $(1 \, 2 \, \ldots \, p)$ on the roots of $g$. 
	    
	    Since the Lie algebra $\n^\Q$ has no abelian factor, there exists either $X_m$ such that $[X_1, X_m] \neq 0$ or there exits $Y_m$ such that $[X_1,Y_m] \neq 0$. 
	    In the first case, the eigenvalue of $[X_1,X_m]$ is $\lambda_1 \lambda_m$ and in the other case, the eigenvalue of $[X_1, Y_m]$ is $\lambda_1 \mu_m$. Since $p$ is prime, both of these eigenvalues must be either fixed by $\sigma$ or have $p$ conjugates under the action of $\sigma$. By Proposition \ref{prop:primeDegFullRank} we know $g$ satisfies the full rank condition and thus if $\lambda_1 \lambda_m$ would be fixed under $\sigma$, Lemma \ref{lem:FullRankPermutationNotTransitive} gives a contradiction. If $\lambda_1 \mu_m$ is fixed by $\sigma$, we get that $\lambda_1 = \ldots = \lambda_p$ since $\mu_m$ is fixed by $\sigma$. This contradicts the fact that $g$ is irreducible and thus the minimal polynomial of $\lambda_1 \lambda_m$ or $\lambda_1 \mu_m$ has degree at least $p$. In particular the statement of the lemma follows for the $j$ for which $\lambda_1 \lambda_m$ or $\lambda_1 \mu_m$ is an eigenvalue of $A_j$.
	\end{proof}
	Before dealing with the final case, we first prove this technical lemma.
	\begin{lemma}
	    \label{lem:deg3deg4ExistenceOfGaloisEl}
	    If $f = g\cdot h$ is a polynomial with $g,h$ irreducible polynomials of degree $p$ and $n$, respectively with $p$ prime and $p$ not a divisor of $n$, then for any root $\lambda$ of $h$ there exists an element of the Galois group of $f$ that fixes $\lambda$ but does not fix any root of $g$.
	\end{lemma}
	\begin{proof}
	    Let $E, E_1$ and $E_2$ denote the splitting fields of $f, g$ and $h$, respectively. By \cite[Theorem 8.4]{stew98-1}, it follows that these are all normal field extensions of $\Q$. For any field automorphism $\sigma \in \gal(E, \Q)$ we thus have that $\sigma(E_1) = E_1$ and $\sigma(E_2) = E_2$ by \cite[Theorem 10.5]{stew98-1}. Therefore we have the natural injection
	    \begin{equation*}
	        \gal(E, \Q) \hookrightarrow \gal(E_1, \Q) \oplus \gal(E_2, \Q): \sigma \mapsto (\sigma|_{E_1}, \sigma|_{E_2})
	    \end{equation*}
	    To see this map is injective, note that $E_1$ and $E_2$ generate $E$ as a field and thus $\sigma|_{E_1}$ and $\sigma|_{E_2}$ completely determine $\sigma$. From here we use this identification to write elements of the Galois group of $f$ as an ordered pair of an element of the Galois group of $g$ and one of the Galois group of $h$. 
	    
	    Let $\lambda$ be one of the roots of $h$. Take an element $\sigma$ of order $p$ in the Galois group of $g$. Since $E/\Q$ is normal, we can extend $\sigma:E_1 \to E_1$ to an element $(\sigma, \tau)$ of the Galois group of $f$ (see \cite[Theorem 10.1]{stew98-1}), where $\tau$ lies in the Galois group of $h$ . The order of $\tau$ can be written as $|\tau| = p^l m$ where $p$ and $m$ are coprime and $l \in \N$. 
	    
	    Now consider the field automorphism $(\sigma, \tau)^m = (\sigma^m, \tau^m)$. Since $p$ and $m$ are coprime and $\sigma$ has order $p$, we get that $\sigma^m$ has order $p$ as well. The order of $\tau^m$ is clearly equal to $p^l$. Let $H$ denote the set of roots of $h$, then the action of $\langle \tau^m \rangle$ on $H$ gives a partition $H = H_1 \sqcup \ldots \sqcup H_k$ where $H_i$ are the orbits under this action. By the orbit stabilizer theorem we get that the size of each orbit $H_i$, must divide $p^l$. If $p$ divides $|H_i|$ for all $1 \leq i \leq k$, then it follows that $p$ divides $\sum_i |H_i| = |H| = n$, which is in contradiction with the assumption. Therefore there must be at least one orbit which counts only one element and thus $\tau^m$ fixes a root of $h$, call it $\mu \in H$. Since $\gal(E_2, \Q)$ acts transitively on $H$, there exists an element $\pi \in \gal(E_2, \Q)$ such that $\pi(\mu) = \lambda$. Again, extend this to a field automorphism $(\pi', \pi)$ of $E$ where $\pi' \in \gal(E_1, \Q)$. Then $(\pi', \pi) (\sigma^m, \tau^m) (\pi', \pi)^{-1}$ fixes the root $\lambda$ and is still of order $p$ when restricted to $E_1$ which means it does not fix any root of $g$.
	\end{proof}
	Combining the previous results leads to the final case for Theorem \ref{thm:noexpanding}.
	\begin{prop}
	    Let $\n^\Q = \n_1 \oplus\ldots\oplus \n_{c}$ be a nilpotent Anosov Lie algebra with $n_1 = 7$ and $\dim \n^\Q < 12$, then $\n$ has a positive grading.
	\end{prop}
	\begin{proof}
	    Let $f_1$ denote the characteristic polynomial of $A_1$, the restriction of $A$ to $\n_1$. We prove the statement separately for the possible factorisations of $f$ in irreducible polynomials.
	    \begin{itemize}
	        \item \underline{$f_1$ is irreducible:} Combine Proposition \ref{prop:primeDegFullRank} and Proposition \ref{prop:n1PrimeFullRankThenDividesni} to find that $7$ divides $n_i$ for all $1 \leq i \leq c$. Together with the assumption $\dim \n < 12$ we must have $c = 1$. In particular $\n^\Q$ has a positive grading.
	        \item \underline{$f_1 = g \cdot h_1 \cdot h_2$ with $\deg g = 3$ and $\deg h_1 = \deg h_2 = 2$:} Since the degree of $g$ is prime and strictly bigger than the degrees of $h_1$ and $h_2$, we have that the degree of $g$ does not divide the order of the Galois groups of $h_1$ and $h_2$. We can assume $\n^\Q$ has no non-trivial abelian factor since in this case our work in the previous sections for smaller $n_1$ and Theorem \ref{thm:directSumAnosov} and \ref{thm:directSumExp} assure $\n^\Q$ has a positive grading. Lemma \ref{lem:productOfPolyWithOnePrime} then gives the existence of an index $1 \leq j \leq c$ such that $n_j \geq 3$. Together with the restriction $\dim \n^\Q < 12$, we get that $c \leq 2$ and thus $\n^\Q$ admits a positive grading.
	        \item \underline{$f_1 = g\cdot h$ with $\deg g = 5$ and $\deg h = 2$:} Since $g$ has prime degree and is strictly greater than the degree of $h$, we have that the degree of $g$ does not divide the order of the Galois group of $h$. We can assume $\n^\Q$ has no non-trivial abelian factor since in this case our work in the previous sections for smaller $n_1$ and Theorem \ref{thm:directSumAnosov} and \ref{thm:directSumExp} assure $\n^\Q$ has a positive grading. Lemma \ref{lem:productOfPolyWithOnePrime} then gives an index $1 \leq j \leq c$ such that $n_j \geq 5$. Together with the restriction $\dim \n^\Q < 12$, we get that $c < 2$ and thus $\n^\Q$ admits a positive grading.
	        \item \underline{$f_1 = g \cdot h$ with $\deg g = 3$ and $\deg h = 4$:} We can assume that $3$ divides the order of the Galois group of $h$, otherwise we can use Lemma \ref{lem:productOfPolyWithOnePrime} in a similar fashion as with the previous two cases and find that $\n^\Q$ must be positively graded. It follows that the Galois group of $h$ is either the alternating group $A_4$ or the whole permutation group $S_4$. Let $\lambda_1, \lambda_2, \lambda_3$ be the roots of $g$ and $\mu_1, \mu_2, \mu_3, \mu_4$ the roots of $h$. Since $\n^\Q$ is not abelian there must be without loss of generality an eigenvector in $\n_2$ with one of following three eigenvalues: $\lambda_1 \lambda_2$, $\lambda_1 \mu_1$ or $\mu_1 \mu_2$. 
	        
	        In case the eigenvalue is $\lambda_1 \lambda_2$, take the element in the Galois group of $g$ corresponding with the permutation $(1\, 2\, 3)$ and extend it to an element $\sigma$ of the Galois group of $f$. Then it is clear that $\sigma(\lambda_1 \lambda_2) = \lambda_2 \lambda_3$ and $\sigma^2(\lambda_1 \lambda_2) = \lambda_3 \lambda_1$ are also eigenvalues of $A_2$. These are three different eigenvalues and thus we get $n_2 \geq 3$. 
	        
	        In case the eigenvalue is equal to $\lambda_1 \mu_1$, we use Lemma \ref{lem:deg3deg4ExistenceOfGaloisEl} and find an element $\sigma$ in the Galois group of $f$ such that $\sigma$ fixes $\mu_1$ and acts by the permutation $(1\, 2\, 3)$ on the roots of $g$. Then $A_2$ must also have eigenvalues $\sigma(\lambda_1\mu_1) = \lambda_2 \mu_1$ and $\sigma^2(\lambda_1 \mu_1) = \lambda_3 \mu_1$. This gives three different eigenvalues of $A_2$ and thus $n_2 \geq 3$. 
	        
	        In case the eigenvalue equals $\mu_1 \mu_2$, take the element in the Galois group of $h$ corresponding to the permutation $(1)(2\, 3\, 4)$ and extend it to an element $\sigma$ of the Galois group of $f$. Then $A_2$ must also have the eigenvalues $\sigma(\mu_1 \mu_2) = \mu_1 \mu_3$ and $\sigma^2(\mu_1, \mu_2) = \mu_1 \mu_4$. This gives three different eigenvalues of $A_2$ and thus $n_2 \geq 3$. 
	        
	        In all cases we see that $n_2 \geq 3$. Therefore by the restriction on the dimension of $\n^\Q$ we find that $c = 2$ and thus that $\n^\Q$ admits a positive grading. 
	   \end{itemize}
    \end{proof}

	\bibliographystyle{plain}
	\bibliography{ref}

\begin{thebibliography}{10}

\bibitem{deki99-1}
Karel Dekimpe.
\newblock Hyperbolic automorphisms and {A}nosov diffeomorphisms on
  nilmanifolds.
\newblock {\em Trans. Amer. Math. Soc.}, 353(7):pp.\ 2859--2877, 2001.

\bibitem{deki18-1}
Karel Dekimpe.
\newblock A users' guide to infra-nilmanifolds and almost-{B}ieberbach groups.
\newblock In {\em Handbook of group actions. {V}ol. {III}}, volume~40 of {\em
  Adv. Lect. Math. (ALM)}, pages 215--262. Int. Press, Somerville, MA, 2018.

\bibitem{dd14-1}
Karel Dekimpe and Jonas Der\'{e}.
\newblock Expanding maps and non-trivial self-covers on infra-nilmanifolds.
\newblock {\em Topol. Methods Nonlinear Anal.}, 47(1):347--368, 2016.

\bibitem{dere13-1}
Jonas Der\'e.
\newblock A new method for constructing {A}nosov {L}ie algebras.
\newblock {\em Trans. Amer. Math. Soc.}, 368(2):1497--1516, 2016.

\bibitem{dere14-1}
Jonas Der\'{e}.
\newblock Gradings on {L}ie algebras with applications to infra-nilmanifolds.
\newblock {\em Groups Geom. Dyn.}, 11(1):105--120, 2017.

\bibitem{dere20-1}
Jonas Der\'{e}.
\newblock Orthogonal bi-invariant complex structures on metric {L}ie algebras.
\newblock {\em Ann. Global Anal. Geom.}, 59(2):157--177, 2021.

\bibitem{dl57-1}
J.~Dixmier and W.G. Lister.
\newblock Derivations of nilpotent {L}ie algebras.
\newblock {\em Proc. Am. Math. Soc.}, 8:155--158, 1957.

\bibitem{fgh13-1}
David~J. Fisher, Robert~J. Gray, and Peter~E. Hydon.
\newblock Automorphisms of real {L}ie algebras of dimension five or less.
\newblock {\em J. Phys. A}, 46(22):225204, 18, 2013.

\bibitem{fran69-1}
J.~Franks.
\newblock Anosov diffeomorphisms on tori.
\newblock {\em Trans. Amer. Math. Soc.}, 145:117--124, 1969.

\bibitem{fran70-1}
J.~Franks.
\newblock Anosov diffeomorphisms.
\newblock {\em Global Analysis: Proceedings of the Symposia in Pure
  Mathematics}, 14,:pp. 61--93, 1970.

\bibitem{grom81-1}
M.~Gromov.
\newblock Groups of polynomial growth and expanding maps.
\newblock {\em Institut des Hautes \'Etudes Scientifiques}, (53):pp. 53--73,
  1981.

\bibitem{laur08-1}
Jorge Lauret.
\newblock Rational forms of nilpotent {L}ie algebras and {A}nosov
  diffeomorphisms.
\newblock {\em Monatsh. Math.}, 155(1):15--30, 2008.

\bibitem{lw08-1}
Jorge Lauret and Cynthia~E. Will.
\newblock On {A}nosov automorphisms of nilmanifolds.
\newblock {\em J. Pure Appl. Algebra}, 212(7):1747--1755, 2008.

\bibitem{lw09-1}
Jorge Lauret and Cynthia~E. Will.
\newblock Nilmanifolds of dimension {$\leq 8$} admitting {A}nosov
  diffeomorphisms.
\newblock {\em Trans. Amer. Math. Soc.}, 361(5):2377--2395, 2009.

\bibitem{mw15-1}
Meera Mainkar and Cynthia~E. Will.
\newblock Characterization of 9-dimensional anosov lie algebras.
\newblock {\em Journal of Lie Theory}, 2015.

\bibitem{main12-1}
Meera~G. Mainkar.
\newblock Anosov {L}ie algebras and algebraic units in number fields.
\newblock {\em Monatsh. Math.}, 165(1):79--90, 2012.

\bibitem{newh70-1}
S.~E. Newhouse.
\newblock On codimension one {A}nosov diffeomorphisms.
\newblock {\em Amer. J. Math.}, 92:761--770, 1970.

\bibitem{payn09-1}
Tracy~L. Payne.
\newblock Anosov automorphisms of nilpotent {L}ie algebras.
\newblock {\em J. Mod. Dyn.}, 3(1):121--158, 2009.

\bibitem{shub69-1}
M.~Shub.
\newblock Endomorphisms of compact differentiable manifolds.
\newblock {\em Amer. J. Math.}, 91,:pp. 175--199, 1969.

\bibitem{shub70-1}
M.~Shub.
\newblock Expanding maps.
\newblock {\em Global Analysis: Proceedings of the Symposia in Pure Mathematics
  XIV}, 14,:pp. 273--277, 1970.

\bibitem{smal67-1}
S.~Smale.
\newblock Differentiable dynamical systems.
\newblock {\em Bull. Amer. Math. Soc.}, 73,:pp. 747--817, 1967.

\bibitem{stew98-1}
Ian Stewart.
\newblock {\em Galois theory}.
\newblock Chapman and Hall, Ltd., London, second edition, 1989.

\end{thebibliography}
	
\end{document}